\documentclass[12pt]{amsart}
\usepackage{amssymb,enumerate,mathrsfs, amsmath, amsthm}
\usepackage{url,titletoc}

\usepackage{subcaption}

\usepackage{setspace}
\usepackage[shortlabels]{enumitem}
\usepackage{xspace}

\usepackage[usenames,dvipsnames]{xcolor}
\definecolor{darkblue}{rgb}{0.0, 0.0, 0.55}
\usepackage[pagebackref,colorlinks,linkcolor=BrickRed,citecolor=OliveGreen,urlcolor=darkblue,hypertexnames=true]{hyperref}
\usepackage{cmap}

\usepackage{adjustbox}
\usepackage{graphicx}
\usepackage{float}
\usepackage{multirow}
\usepackage{cleveref}
\usepackage{pdfpages}
\usepackage[justification=centering]{caption}
\makeindex

\usepackage{listings}
\renewcommand{\qedsymbol}{\rule[.12ex]{1.2ex}{1.2ex}}
\renewcommand{\subset}{\subseteq}

\DeclareMathOperator{\tr}{tr}

\newtheorem{theorem}{Theorem}[section]
\newtheorem{cor}[theorem]{Corollary}
\newtheorem{lemma}[theorem]{Lemma}

\newtheorem{prop}[theorem]{Proposition}
\newtheorem{algo}[theorem]{Algorithm}
\newtheorem{definition}[theorem]{Definition}
\newtheorem{alg}[theorem]{Algorithm}
\newtheorem{remark}[theorem]{Remark}

\newtheorem{spec}[theorem]{Speculation}

\newtheorem{thm}[theorem]{Theorem}

\newtheorem*{lemma*}{Lemma}

\numberwithin{equation}{section}

\def\beq{\begin{equation}}
\def\eeq{\end{equation}}
\def\ben{\begin{enumerate}}
\def\een{\end{enumerate}}

\def\fH{\mathfrak{H}}

\def\cC{\mathcal{C}}
\def\cD{\mathcal{D}}

\def\cH{\mathcal{H}}

\def\cK{\mathcal{K}}

\def\cP{\mathcal{P}}

\def\cT{\mathcal{T}}

\def\cX{\mathcal{X}}

\def\C{\mathbb{C}}

\def\N{\mathbb{N}}

\def\R{\mathbb{R}}
\def\Q{\mathbb{Q}}

\def\eps{\epsilon}

\def\smrg{{SM(\R)^g}}

\def\smxr[#1]{SM_{#1}(\R)}
\def\smxrg[#1]{SM_{#1}(\R)^g}

\def\bs{\bigskip}

\newcommand{\df}[1]{{\bf{#1}}{\index{#1}}}

\def\free{\partial^{\mathrm{free}}}

\def\euc{\partial^{\mathrm{Euc}}}
\def\mat{\partial^{\mathrm{mat}}}

\def\CR{\color{red}}

\def\bill{\color{purple}}

\def\bs{\bigskip}
\def\sec{\section}
\def\ssec{\subsection}
\def\sssec{\subsubsection}

\def\D{\mathcal{D}}
\def\H{\mathcal{H}}

\newcommand{\SM}[2]{SM_{#2}(\R)^{#1}}
\newcommand{\M}[3]{M_{#2 \times #3}(\R)^{#1}}
\def\LA{L_A}
\def\lA{\Lambda_A}

\newcommand{\kron}{\otimes}
\newcommand{\psd}{\succeq}

\newcommand{\set}[1]{\left\{#1\right\}}

\def\vspan{\mathrm{span}}
\newcommand{\Mat}{\textit{Mat}}
\newcommand{\Free}{\textit{Free}}

\newcommand{\mnae}{non-free matrix extreme\xspace}
\newcommand{\Mnae}{Non-free matrix extreme\xspace}
\newcommand{\arvcnt}{Arveson rank-nullity count\xspace}
\newcommand{\euccnt}{Euclidean rank-nullity count\xspace}
\newcommand{\extcnt}{Matrix extreme rank-nullity count\xspace}
\newcommand{\fce}{free Caratheodory expansion\xspace}

\makeatletter
\@namedef{subjclassname@2020}{
    \textup{2020} Mathematics Subject Classification}
\makeatother

\makeatletter
\newcommand{\mycontentsbox}{%
    {\centerline{NOT FOR PUBLICATION}
        \addtolength{\parskip}{-2.3pt}
        \begin{spacing}{0.915}
        \small\tableofcontents
        \end{spacing}}}
\def\enddoc@text{\ifx\@empty\@translators \else\@settranslators\fi
    \ifx\@empty\addresses \else\@setaddresses\fi
    \newpage\mycontentsbox\newpage\printindex}
\makeatother

\title[Matrix and Free Extreme Points of Free Spectrahedra]{Matrix Extreme Points and Free Extreme Points of Free Spectrahedra}
\author[Epperly]{Aidan Epperly}
\address{AE: Mathematics  Department,  UC San Diego, La Jolla, USA}
\email{aepperly@ucsd.edu}

\author[Evert]{Eric Evert}
\address{EE: Computer Science Department, Northwestern University, Evanston, USA}
\email{eric.evert@northwestern.edu}

\author[Helton]{J. William Helton}
\address{JWH: Mathematics  Department,  UC San Diego, La Jolla, USA}
\email{whelton@ucsd.edu}

\author[Klep]{Igor Klep}
\address{
IK: Faculty of Mathematics and Physics, 
University of Ljubljana, 
Slovenia}
\email{igor.klep@fmf.uni-lj.si}
\thanks{IK was supported by the Slovenian Research Agency grants J1-2453, N1-0217 and P1-0222.}

\begin{document}

\date{\today}

\maketitle

\begin{abstract}

    A spectrahedron is a convex set defined as the solution set of a linear matrix inequality, i.e., the set of all $x \in \R^g$ such that
    \[
    \LA(x) = I + A_1 x_1 + A_2 x_2 + \dots + A_g x_g \psd 0
    \]
     for some symmetric matrices $A_1,\ldots,A_g$.
    This notion can be extended to matrix spaces by taking $X$ to be a tuple of real symmetric matrices of any size $n\times n$ and using the Kronecker product
    \[
    \LA(X) = I_n \kron I_d + A_1 \kron X_1 + A_2 \kron X_2 + \dots + A_g \kron X_g.
    \]
    The solution set of $\LA(X) \psd 0$ is called a \textit{free spectrahedron}. Free spectrahedra play an important roll in systems engineering, operator algebras, and the theory of matrix convex sets. Matrix and free extreme points of free spectrahedra are of particular interest.  While many authors have studied matrix and free extreme points of free spectrahedra, it has until now been unknown if these two types of extreme points are actually different.  

    The results of this paper fall into three main categories: theoretical, algorithmic, and experimental. Firstly, we prove the existence of matrix extreme points of free spectrahedra that are not free extreme. This is done by producing exact arithmetic examples of matrix extreme points that are not free extreme.  On the other hand, we show that if the $A_i$ are $2 \times 2$ matrices, then matrix and free extreme points coincide.  Secondly, we detail a number of methods for constructing matrix extreme points of free spectrahedra that are not free extreme, both exactly and approximately (numerically). We also show how a recent result due to Kriel 
    (Complex Anal.~Oper.~Theory 2019)
    can be used to efficiently test whether a point is matrix extreme. Thirdly, we provide evidence, through a series of numerical experiments, that a substantial number of matrix extreme points of free spectrahedra are not free extreme.
 Numerical work in another direction shows how to effectively write a given tuple in a free spectrahedron as a matrix convex combination of its free extreme points.\looseness=-1
\end{abstract}

\newpage

\section{Introduction}

Semidefinite programming \cite{BPR13} is
a generalization of linear programming which has played a profound role in applied mathematics.
It is based on optimization of linear functionals over convex sets defined by linear matrix inequalities, namely, inequalities of the form
$$L_A(X)=I-A_1X_1-\dots-A_g X_g\succeq0.$$
Here the $X_i$ are real numbers and the set of solutions 
is a convex set called a spectrahedron.
Linear programming amounts to the very special case where
every $A_i$ is diagonal.\looseness=-1

These inequalities also make sense when the $X_i$ are symmetric matrices of any size, $n\times n$,
and enter the formula though Kronecker's (tensor) product $A_i\otimes X_i$.
The solution set of $L_A(X)\succeq0$, denoted $\cD_A$, is called a free spectrahedron
since it contains matrices of all sizes and the defining
``linear pencil" is ``free" of  the sizes of the matrices. The subset $\cD_A(n)$ of $\cD_A$ consisting
of $g$-tuples of $n \times n$ matrices $X_i$ is called the 
$n${th} level of $\cD_A$.
Free spectrahedra are important examples of matrix convex sets \cite{EW97,K+}.
 
 Of great importance to the understanding and use of a convex set $\cC$
 is  its extreme points. Reasons for this include:
 \begin{enumerate}[\rm(a)]
     \item 
  the set of optimizers of  a linear functional  over $\cC$ 
 contains an extreme point;
 \item 
  a ``Caratheodory Krein-Milman type'' theorem  which  expresses a given element of a matrix convex set  as a convex combination of  extreme points.
 \end{enumerate}
 
 Free spectrahedra are  level-wise convex, so there is the usual notion 
 of scalar convex combination,
 but  more powerful
 are 
  matrix convex combinations;
  here  contraction matrices summing to the identity play the role of the  coefficients in the combination. 
 These matrix convex combinations allow for {\it convex combinations of matrix tuples of different sizes} thereby leading to a close relationship between the various levels of a matrix convex set. In particular the geometry of the $n$th level of a matrix convex set directly impacts all other levels of the set, thereby making matrix convexity much more restrictive  than classical convexity.

While for classical convex sets there is only one notion of an extreme point,  there are several notions of extreme points for matrix convex sets.
The three natural notions of extreme points are
 {\it Euclidean extreme,
 matrix extreme} and {\it free extreme points}
 with their set containments being 
 \[
   \text{free  extreme} \ \ \subset   \ \ \text{matrix  extreme} \ \subset \ \ \text{Euclidean  extreme}
 \]
 Here 
 Euclidean refers to old-fashioned classical extreme points. This paper focuses on matrix and free extreme points and treats many aspects of them, both theoretical and numerical.

\ssec{The focus of this paper}
Surprisingly, despite recent advancements in the understanding of extreme points of matrix convex sets, it has remained unknown if there are matrix extreme points of a real free spectrahedron that are not free extreme. Though morally speaking such examples should exist as the definition of a matrix extreme point is evidently weaker than the definition of a free extreme point,
for real free spectrahedra, this question has remained open.
For example, extensive experiments
in \cite{EFHY21} optimizing random
linear functionals over random free
spectrahedra found that while
 optimizers were with extremely
 high probability free extreme points the rest were merely
 euclidean extreme or (a few)
 were numerically bad. These experiments did not provide evidence on the open question of the existence of matrix extreme points that are not free extreme.

In this paper we 
exhibit
real free spectrahedra 
in three and in four variables
with matrix extreme points that are not free extreme.
A major difficulty is that this
involves big matrices and
the computations must be done algebraically.
In another direction we prove that if the
$A_j$ are $2 \times 2$ matrices, then matrix and free extreme points of $\cD_A$ are the same.

In addition we report systematic numerical
experiments (for free spectrahedra in $g =2,3,4$ variables) aimed at two goals:

\ben[\rm(1)]
\item
indicating if 
matrix extreme points which are not free extreme are rare,

\item 
developing an effective 
algorithm for expanding a point in a free spectrahedron as a matrix convex combination of free or matrix extreme points, that is, produce  a 
free Caratheodory expansion.
\een 

Working on these two objectives together is natural since the underlying algorithms are close to each other.
Of independent interest is that to meet these objectives we develop a new technique called Nullspace Purification,
which when added to the underlying algorithms greatly improves their accuracy.

The conclusions of our experiments are as follows.
There is a strong indication that matrix but not free  extreme points are not rare. 
In addition, the free Caratheodory  expansion with Nullspace Purification had a low failure rate, especially for $g>2$. Moreover, the addition of Nullspace purification to the previously best known approach for computing a free Caratheodory expansion \cite{EFHY21} significantly improved the success rate.
In addition, the ``size" of the expansion 
was considerably lower than the provable 
worst case. This illustrates that free Caratheodory expansions in practice are typically even nicer than than theory suggests. See \Cref{sec:free-carath}
for further comments on the observed patterns. 

For quantitative versions of the assertions just made see
the section on Conclusions,
\Cref{section:conclusion}.
Further observations on patterns are at the end of \Cref{sec:experiments} and \ref{sec:free-carath}
and tables of data in these sections should be  interesting to enthusiasts.

The experiments reported on in this article where run using NCSE (NonCommutative Spectrahedron Extreme) \cite{EOYH21}, an NCAlgebra \cite{OHMS17} package\footnote{Both NCSE and NCAlgebra are written using Mathematica.} for performing computations with extreme points of free spectrahedra. In addition, to running our experiments, we added  a number of new functions to NCSE. This includes algorithms for the aforementioned Nullspace Purification technique and algorithms for determining if a given element of a free spectrahedron is a matrix extreme point.

\ssec{Context and related work}
A major movement over the last few decades is to extend theorems involving polynomials,
rational functions  and power series expansions 
on $\C^g$ and $\R^g$ to noncommuting
analogs acting on $g$ tuples of matrices or operators
leading to the development of the booming area of
{\it free analysis}
\cite{BGM06,Pop06,dOHMP09,Voi10,KVV14,AM15,AJP20,JMS21}.
Early in this direction was free probability which 
ultimately applied to random matrices of large size, cf.
\cite{MS17}.
A later avenue is 
free real algebraic geometry which
extends real algebraic geometry from $\R^g$ to
g-tuples of matrices, \cite{HM12,NT15,BKP16,K+}.

Another major direction in free analysis involves 
Semidefinite programming (SDP) and LMIs
as we are studying here.
Related structure  whose techniques are very useful 
for free LMIs have been studied by pure mathematicians
for decades in 
 the area of operator theory. 
 Extreme points in this context, often called Arveson boundary points,
go back to
Arveson's seminal work \cite{A69} and have since been studied by many authors, e.g., see \cite{Ham79,Agl88,MS98,DM05,A08,KLS14,DK15,FHL18,Pas22}. 
The overt introduction of matrix extreme points was in
\cite{WW99} with many advances at the same time due to
Farenick \cite{F00,F04}.
Recently various  works used  free spectrahedra to 
study spectrahedral inclusion
\cite{HKM13,DDSS17,FNT17,Zal17,HKMS19}, and
profound implications for quantum information theory
were exposited in \cite{BN18,BJN22}.
 A booming area based on ``noncommutative optimization'' \cite{DLTW08,PNA10,BKP16,MBM21,WM21}
 is quantum game theory \cite{GLL19,CDN20,PR21}.
Free LMIs
serve as a model for convex structure  which occurs in
linear control and systems engineering problems
specified entirely by signal flow diagrams, cf.~\cite{HMPV09}.

 The remainder of the introduction turns to details,
 giving precise definitions and background theorems.

\def\bX{X}

\def\ot{{\otimes}}

\def\sph{spectrahedr}

\ssec{Matrix notation}
We let $\M{g}{m}{n}$ denote the set of $g$-tuples of $m\times n$ matrices with real entries, and $M_{n}(\R)=M_{n\times n}(\R)$.
Similarly, let $\SM{g}{n}$ denote the set of all $g$-tuples of real symmetric $n\times n$ matrices $X_i$ and set $SM(\R)^g = \cup_n\SM{g}{n}$. 

A matrix $B \in M_{n}(\R)$ is said to be \df{positive semidefinite} if it is symmetric, i.e., $B = B^T$, and all of its eigenvalues are nonnegative. Let $B\psd 0$ denote that the matrix $B$ is positive semidefinite. Similarly, given two symmetric matrices $B_1,B_2 \in M_{n}(\R)$, let $B_1 \psd B_2$ denote that $B_1 - B_2$ is positive semidefinite.

\ssec{Free spectrahedra and linear matrix inequalities}

In this paper, we primarily concern ourselves with a specific class of 
convex sets called free spectrahedra. A free spectrahedron is a matrix convex set that can be defined by a linear matrix inequality. Fix a tuple $A\in\SM{g}{d}$ of $d\times d$ symmetric matrices. A \df{monic linear pencil} $\LA(x)$ is a sum of the form 
\[
\LA(x) = I_d + A_1 x_1 + A_2 x_2 + \dots + A_g x_g.
\]
Given a tuple $X \in \SM{g}{n}$, the \df{evaluation} of $\LA$ at $X$ is 
\[
\LA(X) = I_{d} \kron I_n + A_1 \kron X_1 + A_2 \kron X_2 + \dots + A_g \kron X_g
\]
where $\kron$ denotes the Kronecker Product. A \df{linear matrix inequality} is an inequality of the form
\[
\LA(X) \psd 0.
\]
Let $\lA(X)$ denote the homogeneous linear part of $\LA(X)$, i.e.,
\[
\lA(X) = A_1 \kron X_1 + A_2 \kron X_2 + \dots + A_g \kron X_g,
\]
so that $\LA(X) = I_{dn} + \lA(X)$.

Given a $g$-tuple $A\in\SM{g}{n}$ and a positive integer $n$, we define the \df{free spectrahedron $\D_A$ at level $n$}, denoted $\D_A(n)$, by
\[
\D_A(n) := \{X\in \SM{g}{n} : \LA(X) \psd 0\}.
\]
That is, $\D_A(n)$ is the set of all $g$-tuples of $n\times n$ real symmetric matrices $X$ such that the evaluation $\LA(X)$ is positive semidefinite. Define the \df{free spectrahedron $\D_A$} to be the union over all $n$ of the free spectrahedron $\D_A$ at level $n$, i.e.
\[
\D_A := \bigcup_{n=1}^{\infty}\D_A(n)\subseteq\SM{g}{}.
\]

We say a free spectrahedron is \df{bounded} if there is some real number $C$ so that
\[
CI_n - \sum_{i=1}^gX_i^2 \psd 0
\]
for all $X = (X_1, X_2,\dots,X_g) \in \D_A(n)$ and all positive integers $n$. It is routine to show that a free spectrahedron is bounded if and only if $\D_A(1)$ is bounded \cite{HKM13}. In our definition of a free spectrahedron, we use a non-strict inequality. All free spectrahedra defined in this way are closed in the sense that each $\D_A(n)$ is closed.

\sssec{Minimal Defining Tuples}

Throughout this paper, the size of the defining tuple of a free spectrahedron will play an important role in our analysis. However, we must note that the defining tuple of a free spectrahedron is not unique. For instance, $A\in \SM{g}{d}$ and $A \oplus A$ both define the same free spectrahedron despite $A\oplus A$ being in $\SM{g}{2d}.$ Thus, we would not necessarily expect the size of the defining tuple of a spectrahedron to be an inherent property of the free spectrahedron. We can, however, overcome this with the concept of a \df{minimal defining tuple}. Using \cite{HKM13}, we define a minimal defining tuple $\widetilde{A} \in \SM{g}{d}$ of a free spectrahedron $\D_A$ as a tuple of minimal size such that $\D_{\widetilde{A}} = \D_A$. That is to say if $\widetilde{A} \in \SM{g}{d}$ is a minimal defining tuple of $\D_A$ and $\widehat{A} \in \SM{g}{n}$ for $n < d$, then $\D_{\widehat{A}} \neq \D_{\widetilde{A}}.$ 
The minimal defining tuple $\widetilde A$
is unique up to unitary equivalence \cite[Theorem 3.12 and Corollary 3.18]{HKM13}.  Here, two tuples $X,Z\in\Gamma(n)$ are said to be \df{unitarily equivalent} if there exists some unitary matrix $U$ such that $X_j = U^T Z_j U.$
Throughout the rest of the paper, we will always assume that the defining tuple of a free spectrahedron is minimal.

\sssec{Homogeneous free spectrahedra}
\label{sec:HomogeneousSpectrahedraDefs}

Our proofs make heavy use of homogeneous free spectrahedra and  of a canonical association between free spectrahedra and homogeneous free spectrahedra. 
 Given a tuple $(A_0, \dots, A_g) = (A_0,A) \in \SM{g+1}{d}$ the \df{homogeneous free spectrahedron $\H_{(A_0,A)}$ at level $n$}, denoted $\H_{(A_0,A)}(n)$, is given by
\[
\H_{(A_0,A)}(n) := \set{(X_0,\dots,X_g)\in\SM{g+1}{n}:\Lambda_{(A_0,A)} (X_0,X) \psd 0}
\]
and the \df{homogeneous free spectrahedron $\H_{(A_0,A)}$} is defined as the union over all positive integers $n$ of $\H_{(A_0,A)}(n)$, i.e.,
\[
\H_{(A_0,A)} = \cup_{n=1}^{\infty} \H_{(A_0,A)}(n) \subseteq SM(\R)^{g+1}
\]

We write the tuples $(X_0,\dots,X_g)$ and $(A_0,\dots,A_g)$ as $(X_0,X)$ and $(A_0,A)$, respectively, since we will frequently pass from nonhomogeneous free spectrahedra $\cD_A$ with elements $X \in \SM{g}{}$ to homogeneous free spectrahedra $\cH_{(I,A)}$ with elements $(X_0,X) \in \SM{g+1}{}$. We call $X_0$ and $A_0$ the \df{inhomogeneous component}  of the corresponding tuple.

In particular, given a free spectrahedron $\cD_A$ in $g$ variables, the \df{homogenization} of $\cD_A$, denoted $\fH (\cD_A)$, is the homogeneous free spectrahedron 
\[
\fH(\cD_A) :=\cH_{(I,\check{A})}
\]
where $\check{A}$ is any minimal defining tuple for $\cD_A$. Here we emphasize the minimality of $\check{A}$, as it is required for the homogenization to be well-defined, see 
\cite{E21}.  On the other hand, given a homogeneous free spectrahedron $\cH_{(A_0,A)}$, let $\fH^{-1} (\cH_{(A_0,A)})$ denote the set
\[
\fH^{-1} (\cH_{(A_0,A)}) = \{X \in \smrg | \ (I,X) \in \cH_{(A_0,A)} \}.
\]
A homogeneous free spectrahedron $\cH_{(A_0,A)}$ is \df{sectionally bounded} if $\fH^{-1} (\cH_{(A_0,A)})$ is bounded. We say that $\cH_{(A_0,A)}$  is a \df{positive homogeneous free spectrahedron} if it contains $(1,0,\dots,0) \in SM_1(\R)^{g+1}$ in the interior of $\cH_{(A_0,A)} (1)$. Note that the homogenization $\fH(\cD_A)$ of any free spectrahedron $\cD_A$ is always a positive homogeneous free spectrahedron since $\cD_A$ always contains $0 \in SM_1(\R)^g$ in its interior.

We extend the definitions of $\fH$ and $\fH^{-1}$ to matrix tuples as follows. Given a tuple $X \in SM (\R)^g$  define
\[
\fH(X)=(I,X) \in SM(\R)^{g+1}.
\]
On the other hand, given a matrix tuple $(X_0,X) \in SM (\R)^{g+1}$ with $X_0 \succeq 0$, define
\[
\fH^{-1} (X) = X_0^{\dagger/2} X X_0^{\dagger/2} \in \smrg.
\]
Here $X_0^{\dagger/2}$ denotes the positive semidefinite square root of the Moore-Penrose pseudoinverse of $X_0$. We refer the reader to \cite{E21,K+} for further discussion of homogeneous free spectrahedra. 

\ssec{Matrix Convex Sets}

Given some finite collection $\set{X^i}_{i=1}^\ell$ with $X^i \in \SM{g}{n_i}$ for each $i = 1,2,\dots,\ell$, a \df{matrix convex combination} of $\set{X^i}_{i=1}^\ell$ is a sum of the form
\[
\sum_{i=1}^{\ell} V_i^T X^i V_i \qquad \text{with} \qquad \sum_{i=1}^{\ell} V_i^T V_i = I_n
\] 
where $V_i \in \M{}{n_i}{n}$ and 
\[
V_i^T X^i V_i = \left(V_i^T X_1^i V_i, V_i^T X_2^i V_i,\dots,V_i^T X_g^i V_i \right) \in \SM{g}{n}
\] 
for all $i = 1,2,\dots,\ell$. We emphasize that the tuples $X^i$ need not be the same size. A matrix convex combination is called \df{proper} if $V_i$ is surjective for all $i = 1,2,\dots,\ell$.

A set $\Gamma \subseteq SM(\R)^g$ is \df{matrix convex} if it is closed under matrix convex combinations. The \df{matrix convex hull} of a set $\Gamma \subseteq SM(\R)^g$ is the set of all matrix convex combinations of the elements of $\Gamma$. Equivalently, the matrix convex hull of $\Gamma \subseteq SM(\R)^g$ is the smallest matrix convex set containing $\Gamma$. That is, the matrix convex hull of $\Gamma$ is the intersection of all matrix convex sets containing $\Gamma$.

A set $\Gamma \subseteq SM(\R)^g$ is a \df{matrix cone} if given any finite collection $\set{X^i}_{i=1}^\ell$ with $X^i\in \SM{g}{n_i}$ for each $i = 1,2,\dots,\ell$ and $V_{n_i}\in M_{n_i \times n}(\R)$, then
\begin{align*}
    \sum_{i=1}^\ell V_i^T X^i V_i \in \Gamma.
\end{align*}

\begin{lemma}
Let $A\in \SM{g}{d}$ and let $\D_A$ be the associated free spectrahedron. Then $\D_A$ is matrix convex.
\end{lemma}

\begin{proof}
The matrix convexity of free spectrahedra follows quickly from the fact that $I_{dn} + \lA(X) \psd 0$ implies that 
\[
\begin{split}
0 & \preceq (I_d\otimes V)^T (I_{d}\otimes I_n + \lA(X) )(I_d\otimes V) =
I_{d} \kron V^T V + \lA(V^T X V)  \\
& = I_{dn}+ \lA(V^T X V) 
.\qedhere
\end{split}
\]
\end{proof}

\ssec{Extreme points of matrix convex sets}
The extreme points of matrix convex sets and free spectrahedra are of particular interest, since they have Krein-Milman type spanning properties \cite{WW99,K+,EH19}. The paper will primarily consider three types of extreme points: Euclidean extreme points, matrix extreme points, and free extreme points.

Given a matrix convex set $\Gamma$, we say $X \in \Gamma(n)$ is a \df{Euclidean extreme point} of $\Gamma$ if $X$ cannot be written as a nontrivial classical convex combination of points in $\Gamma(n)$. We note that this is the same as being a classical extreme point of $\Gamma(n)$. We let $\euc{\Gamma}$ denote the set of all the Euclidean extreme points of $\Gamma$.

 We say a point $X\in \Gamma(n)$ is a \df{matrix extreme point} of $\Gamma$ if whenever $X$ is written as a proper matrix convex combination 
\[
X = \sum_{i=1}^\ell V_i^T X^i V_i\qquad \text{with} \qquad \sum_{i=1}^\ell V_i^T V_i = I_n
\]
of points $X^i \in \Gamma$ for $i = 1,2,\dots,\ell$, then for every $i = 1,2,\dots,\ell,$ we have $V_i \in M_{n}(\R)$ and $X$ is unitarily equivalent to $X^i$. We let $\mat{\Gamma}$ denote the set of all the matrix extreme points of $\Gamma$.

Finally, we say a point $X \in \Gamma(n)$ is a \df{free extreme point} of $\Gamma$ if whenever $X$ is written as a matrix convex combination 
\[
X = \sum_{i=1}^\ell V_i^T X^i V_i\qquad \text{with} \quad \sum_{i=1}^\ell V_i^T V_i = I_n
\]
of points $X^i \in \Gamma$ with $V_i \neq 0$ for each $i$, then for all $i = 1,2,\dots,\ell$ either $V_i \in M_{n}(\R)$ and $X$ is unitarily equivalent to $X^i$ or $V_i \in \M{}{n_i}{n}$ where $n_i > n$ and there exists $Z^i\in \Gamma$ such that $X \oplus Z^i$ is unitarily equivalent to $X^i$. In words, a point $X$ is a free extreme point of $\Gamma$ if it cannot be written as a nontrivial matrix convex combination of points in $\Gamma$. We let $\free{\Gamma}$ denote the set of all the free extreme points of $\Gamma$.

\sssec{Irreducible matrix tuples}

Given a matrix $M \in M_{n}(\R)$, a subspace $N \subseteq \R^n$ is a \df{reducing subspace} if both $N$ and $N^\perp$ are invariant subspaces of $M$, which is to say that $N$ is a reducing subspace of $M$ if $MN \subseteq N$ and $MN^\perp \subseteq N^\perp$. A tuple $X\in \SM{g}{n}$ is \df{irreducible} (over $\R$) if the matrices $X_1, \dots, X_g$ have no common reducing subspaces in $\R^n$; a tuple is \df{reducible} (over $\R$) if it is not irreducible. Since we will always be working over $\R$ in this paper, we will drop the use of ``over $\R$'' and simply refer to tuples as reducible or irreducible.

\begin{remark}
Given a matrix convex set $\Gamma$, if $X\in \Gamma$ is matrix extreme, then it is straightforward to show that $X$ is irreducible. 
\end{remark}

\ssec{Significance of free and matrix extreme points}
\label{sec:FreeExtSignificance}

The concept of matrix extreme points was introduced by Webster and Winkler in \cite{WW99}. It is known that the matrix convex hull of the matrix extreme points of a closed bounded matrix convex set $\Gamma$ is equal to $\Gamma$ \cite{K+}. In particular, this means that every point in a free spectrahedron can be written as a matrix convex combination of matrix extreme points.
In \cite{EH19} it was shown that every bounded free spectrahedron $\D_A$ is the matrix convex hull of its free extreme points. Moreover, the free extreme points are the smallest set with this property. We call an expansion of an element of a free spectrahedron in terms of free extreme points a \df{free Caratheodory expansion.}

\begin{theorem}[{\cite[Theorem 1.1]{EH19}}]\label{theorem:free-ext-small}
Let $A\in\SM{g}{d}$ such that $\D_A$ is a bounded free spectrahedron. Then $\D_A$ is the matrix convex hull of its free extreme points. Furthermore, if $E\subset \D_A$ is a set of irreducible tuples which is closed under unitary equivalence and whose matrix convex hull is equal to $\D_A$, then $E$ must contain the free extreme points of $\D_A.$
\end{theorem}

Thus, the free extreme points are, in some sense, the correct notion of extreme points for free spectrahedra, that is, the point $X$ in the free spectrahedron $\D_A$ has a \fce. 
 There are some differences when one considers these notions over $\C$ instead of over $\R$, cf.~\cite{Pas22}.

\ssec{Guide to Main Results}

The paper is organized as follows.
\Cref{section:theory} focuses on the background theory underpinning our results. Of note is \Cref{theorem:Kriel}, a result of Kriel \cite{K+}, which allows us to characterize matrix extreme points. \Cref{sec:d=2} proves that the free and matrix extreme points of $\cD_A$ are the same when $A \in SM_2 (\R)^g$. 
In \Cref{section:exact-points} we exhibit examples of tuples that are matrix extreme points but not free extreme. 
\Cref{section:igors-method} and \Cref{sec:algorithms}  describe algorithms for generating exact and numerical extreme points respectively. 
 \Cref{sec:algorithms} also describes an algorithm for computing the free Caratheodory expansion of a given $X^0 \in \cD_A$,
 as well as a technique called Nullspace Purification which greatly
 improves  accuracy of our experiments.
\Cref{sec:experiments} focuses heavily on empirical observations resulting from a number of numerical experiments; these experiments lead us to believe that \mnae points are not rarities. 
\Cref{sec:free-carath} tests our free Caratheodory algorithm  together with Nullspace Purification and establishes that it is successful.
\Cref{section:conclusion} summarizes our findings.

\sec{Background Theory}\label{section:theory}

\ssec{Characterizing Extreme Points}

An important class of extreme points that has yet to be discussed is the \df{Arveson extreme points} which arise from dilation theory. 
Given a point $X \in \SM{g}{n}$, we say $Y \in \SM{g}{n+k}$ is a \df{dilation}, or more specifically a \df{k-dilation}, of $X$ if $Y$ is of the form
\[
Y = \begin{pmatrix}
X & \beta \\ \beta^T & \gamma
\end{pmatrix}
\]
for $\beta \in M_{n \times k}(\R)^g$ and $\gamma \in \SM{g}{k}. $

A point $X \in \Gamma$ is an Arveson extreme point of $\Gamma$ if 
\begin{align*}
    Y = \begin{pmatrix} X & \beta \\ \beta^T & \gamma \end{pmatrix} \in \Gamma
\end{align*}
implies $\beta = 0$ for $\beta \in M_{n\times 1}(\R)^g$. 
 In words, $X$ is an Arveson extreme points of $\Gamma$ if the only 1-dilations $Y \in \Gamma$ of $X$ are trivial. We can  provide a similar characterization of a Euclidean extreme point in the language of dilation theory:

\begin{prop}[{\cite[Corollary 2.3]{EHKM18}}]
A point $X$ in a free spectrahedron $\D_A$ is a Euclidean extreme point of $\D_A$ if and only if 
\begin{align*}
    Y = \begin{pmatrix} X & \beta \\ \beta & \gamma \end{pmatrix} \in \D_A
\end{align*}
implies $\beta = 0$ for $\beta \in \SM{g}{n}$.
\end{prop}

A theorem of Kriel gives a similarly flavored result regarding matrix extreme points. We include a simple proof for the sake of completeness. 

\begin{theorem}[{\cite[Theorem 6.5.c]{K+}}]\label{theorem:Kriel}
A point $X$ in a bounded free spectrahedron $\D_A$ is a matrix extreme point of $\D_A$ if and only if $(I,X)$ is on a classical extreme ray of $\H_{(I,A)}.$
\end{theorem}

\begin{proof}
For the forward direction, suppose $X\in\D_A(n)$ is a matrix extreme point of the bounded free spectrahedron $\D_A$. By \cite[Lemma 2.2]{E21}  and
\cite[Theorem 6.5.b]{K+} every non-zero element of $\H_{(I,A)}(k)$ can be written in the form $(V^T V, V^T Y V)$ for $Y\in \D_A(k_0)$ and surjective $V\in M_{k_0\times k} (\R)$. For $i=1,2,\dots,m$, fix $Y^i\in \D_A(n_i)$, surjective $V_i\in M_{n_i\times n}(\R)$, $\alpha_i > 0$ such that
\begin{align*}
    (I,X) = \sum_{i=1}^m \alpha_i (V_i^T V_i, V_i^T Y^i V_i)
\end{align*}
and $(V_i^T V_i, V_i^T Y^i V_i) \in \H_{(I,A)}.$ Letting $U_i = \sqrt{\alpha_i} V_i$, we get
\begin{align*}
    (I,X) = \sum_{i=1}^m (U_i^T U_i, U_i^T Y^i U_i) = (\sum_{i=1}^m U_i^T U_i,\sum_{i=1}^m U_i^T Y^i U_i),
\end{align*}
and thus $X = \sum_{i=1}^m U_i^T Y^i U_i$ is a proper convex combination. 
Since $X$ is assumed to be matrix extreme,
$n_i = n$ and there is unitary $W_i \in M_{n\times n}(\R)$  and $\lambda_i > 0$ such that $X = W_i^T Y^i W_i$ and $U_i = \lambda_i W_i$ for $i = 1,2,\dots,m$. So $(V_i^T V_i,V_i^T Y^i V_i)$ is a scalar multiple of $(I,X)$.

For the reverse direction, suppose $X\in \D_A(n)$ such that $(I,X)$ is on a classical extreme ray of $\H_{(I,A)}$ and there are $Y^i \in \D_A(n_i)$ and surjective $V_i \in M_{n_i\times n}$  for $i=1,2,\dots,m$ with $I = \sum_{i=1}^m V_i^T V_i$ and $X = \sum_{i=1}^m V_i^T Y^i V_i.$ 
Thus
\begin{align*}
    (I,X) = (\sum_{i=1}^m V_i^T V_i,\sum_{i=1}^m V_i^T Y^i V_i) = \sum_{i=1}^m (V_i^T V_i, V_i^T Y^i V_i).
\end{align*}
As $(I,X)$ is on a classical extreme ray, $(V_i^T V_i, V_i^T Y^i V_i) = \alpha_i (I,X)$ for some $\alpha_i > 0$. In particular $V_i^T V_i = \alpha_i I$, so $W_i = (\alpha_i^{-\frac{1}{2}}) V_i$ is unitary, $V_i$ is surjective, and $X$ is unitarily equivalent to $Y^i$ for $i=1,2,\dots,m$.
\end{proof}

\begin{cor}\label{lemma:Kriel2}
For a point $X$ in a bounded free spectrahedron $\D_A$ the following are equivalent:
\begin{enumerate}[\rm (i)]
\item \label{it:matext} $X$ is a matrix extreme point of $\D_A$;
\item \label{it:extRay}
$(I,X)$ is on an classical extreme ray of $\H_{(I,A)}$;
\item \label{it:matKerCont}
$(\beta_0,\beta)\in SM_n(\R)^{g+1}$ and
\[
\ker\Lambda_{(I,A)}(I,X) \subseteq \ker\Lambda_{(I,A)}((\beta_0,\beta)) \implies (\beta_0,\beta)\in\vspan((I,X));
\]
\item \label{it:matDilation}
For $(\beta_0,\beta),(\gamma_0,\gamma)\in\SM{g+1}{n}$,
\begin{align*}
    (Y_0,Y) = \begin{pmatrix} (I,X) & (\beta_0,\beta) \\ (\beta_0,\beta) & (\gamma_0,\gamma) \end{pmatrix} \in \H_{(I,A)}
\end{align*}
implies $(\beta_0,\beta) \in \vspan((I,X))$.
\end{enumerate}
\end{cor}

\begin{proof}
The equivalence of \Cref{it:matext} and \Cref{it:extRay} follows from Theorem \ref{theorem:Kriel}. The equivalence of \Cref{it:extRay} and \Cref{it:matKerCont} is \cite[Corollary 4]{RG95} by viewing $(I,X)$ as an element of $\R^{\frac{n(n+1)(g+1)}{2}}$.

To show the equivalence of \Cref{it:matKerCont} and \Cref{it:matDilation}, let $(\beta_0,\beta),(\gamma_0,\gamma)\in\SM{g+1}{n}$ such that
\begin{align*}
    (Y_0,Y) = \begin{pmatrix} (I,X) & (\beta_0,\beta) \\ (\beta_0,\beta) & (\gamma_0,\gamma) \end{pmatrix}.
\end{align*}
By conjugating by permutation matrices, sometimes called canonical shuffles, we see that $\Lambda_{(I,A)}((Y_0,Y))$ is unitarily equivalent to
 \begin{align} \label{eq:ShuffledPencilEval}
    \begin{pmatrix}
    \Lambda_{(I,A)}((I,X)) & \Lambda_{(I,A)}((\beta_0,\beta)) \\
    \Lambda_{(I,A)}((\beta_0,\beta)) & \Lambda_{(I,A)}((\gamma_0,\gamma))
    \end{pmatrix}.
\end{align}
 A routine calculation using the above then shows that  $(Y_0,Y) \in \cH_{(I,A)}$ implies 
\begin{align*}
    \ker \Lambda_{(I,A)} ((I,X)) \subseteq \ker \Lambda_{(I,A)} ((\beta_0,\beta)).
\end{align*}
It follows that \Cref{it:matKerCont} implies \Cref{it:matDilation}. 

Now assume that \Cref{it:matDilation} holds. Taking the Schur complement of the matrix in \Cref{eq:ShuffledPencilEval} shows that $\Lambda_{(I,A)}((Y_0,Y)) \psd 0$ if and only if
\begin{align} \label{eq:Schur}
    &\Lambda_{(I,A)}((I,X))-\Lambda_{(I,A)}((\beta_0,\beta))(\Lambda_{(I,A)}((\gamma_0,\gamma)))^\dagger \Lambda_{(I,A)}((\beta_0,\beta)) \psd 0
\end{align}
and
\begin{align} \label{eq:SchurGamma}
    \Lambda_{(I,A)}((\gamma_0,\gamma)) \psd 0
\end{align}
where $\dagger$ denotes the Moore-Penrose pseudoinverse. 
Fix $(Z_0,Z) \in \SM{g+1}{n}$ such that
\[
\ker\Lambda_{(I,A)} ((I,X)) \subseteq \ker\Lambda_{(I,A)} ((Z_0,Z)).
\]
Then, considering  \Cref{eq:Schur} and \Cref{eq:SchurGamma} shows that there is some $\alpha > 0$ such that for $(\gamma_0,\gamma) = (I,0)$ and $(\beta_0,\beta) = \alpha (Z_0,Z)$ we have $(Y_0,Y) \in\H_{(I,A)}.$ 
By assumption, $(\beta_0,\beta)$ must be in the span of $(I,X)$ and thus,  $(Z_0,Z)$ is too.
\end{proof}

Using these characterizations of Euclidean and matrix extreme points, we arrive at the following known result, see \cite[Theorem 1.1]{EHKM18}.

\begin{prop}
Let $\D_A$ be a bounded free spectrahedron.
\begin{enumerate}[\rm (1)]
    \item \label{it:free-ext-is-irr-arv}
    A tuple $X$ is a free extreme point of $\D_A$ if and only if $X$ is an irreducible Arveson extreme point of $\D_A$.
    \item \label{it:free-ext-imp-mat-ext}
    If a tuple $X$ is a free extreme point of $\D_A$, then $X$ is a matrix extreme point of $\D_A$.
    \item \label{it:mat-ext-imp-euc-ext}
    If a tuple $X$ is a matrix extreme point of $\D_A$, then $X$ is a Euclidean extreme point of $\D_A$.
    \item \label{it:arv-ext-imp-euc-ext}
    If a tuple $X$ is an Arveson extreme point of $\D_A$, then $X$ is a Euclidean extreme point of $\D_A$.
\end{enumerate}
\end{prop}

\begin{proof}
 \Cref{it:free-ext-is-irr-arv} and \Cref{it:arv-ext-imp-euc-ext} are the subject of \cite[Theorem 1.1]{EHKM18}, where the proof is given working over $\C$.  The proof of \Cref{it:arv-ext-imp-euc-ext} can be used over $\R$ without modification, and the proof of \Cref{it:free-ext-is-irr-arv} over $\R$ is given by \cite[Theorem 1.2]{EH19}. \Cref{it:free-ext-imp-mat-ext} follows from \cite[Theorem 1.1]{EH19} which is given as \Cref{theorem:free-ext-small} here. \Cref{it:mat-ext-imp-euc-ext} follows from the observation that if $X$ can be written as a nontrivial classical convex combination $X = \alpha_1 X^1 + \alpha_2 X^2 + \dots + \alpha_\ell X^\ell$, then $(I,X) = \alpha_1 (I,X^1) + \alpha_2 (I,X^2) + \dots + \alpha_\ell (I,X^\ell)$ is a nontrivial classical convex combination of points in $\H_{(I,A)}$.
\end{proof}

\begin{cor}
Let $\D_A$ be a bounded free spectrahedron. If $X$ is a matrix extreme point of $\D_A$, then $X$ is a free extreme point of $\D_A$ if and only if $X$ is an Arveson extreme point of $\D_A$.
\end{cor}

\begin{proof}
A point in $\D_A$ is free extreme if and only if it is irreducible and Arveson extreme. If $X$ is a matrix extreme point of $\D_A$, then in particular $X$ is irreducible and thus $X$ is free extreme if and only if $X$ is Arveson extreme.
\end{proof}

Thus, if $X$ is a matrix extreme point of $\D_A$, to show that $X$ is not free extreme it is sufficient to show it is not Arveson extreme. 
Checking if a tuple is Arveson extreme is straightforward, indeed it is equivalent to solving the upcoming linear system \Cref{eq:arv}.

\sssec{Extreme points and linear systems}

It is possible to determine if a point $X$ in the free spectrahedron $\D_A$ is extreme by solving a linear system. Given a free spectrahedron $\D_A$ and a point $X\in \D_A$ we let
\[
k_{A,X} := \dim \ker \LA(X)
\]
and let $K_{A,X}$ 
be an $nd \times k_{A,X}$ matrix whose columns form an orthonormal basis for the kernel of $\LA(X).$

\begin{thm} \label{theorem:Equations}
Let $\D_A$ be a bounded free spectrahedron and $X \in \D_A(n)$.
\begin{enumerate}[\rm (1)]
    \item \label{it:ArvEq}
    $X$ is an Arveson extreme point of $\D_A$ if and only if the only solution to the homogeneous linear equations
    \begin{align} \label{eq:arv} \index{Arveson Equation}
        \lA(\beta^T) K_{A,X} = (A_1 \kron \beta_1^T + \dots + A_n\kron\beta_g^T) K_{A,X} = 0
    \end{align}
    in the unknown $\beta \in M_{n\times 1}(\R)^g$ is $\beta=0$.
    \item \label{it:EucEq}
    $X$ is a Euclidean extreme point of $\D_A$ if and only if the only solution to the homogeneous linear equations
    \begin{align} \label{eq:euc} \index{Euclidean Equation}
        \lA(\beta) K_{A,X} = (A_1 \kron \beta_1 + \dots + A_n\kron\beta_g) K_{A,X} = 0
    \end{align}
    in the unknown $\beta \in \SM{g}{n}$ is $\beta=0$.
    \item \label{it:MatEq}
    $X$ is a matrix extreme point of $\D_A$ if and only if the only solution to the homogeneous linear equations
    \begin{align} \label{eq:mat} \index{Matrix Extreme Equation}
        \Lambda_{(I,A)}(\beta_0,\beta) K_{A,X} &= (I\kron\beta_0 + A_1 \kron \beta_1 + \dots + A_n\kron\beta_g) K_{A,X} = 0 \\
        \label{eq:matPerp}
        \langle (I,X),(\beta_0,\beta) \rangle &= \tr(\beta_0 + X_1^T \beta_1 + \dots + X_g^T \beta_{g}) = 0
    \end{align}
    in the unknown $(\beta_0,\beta) \in \SM{g+1}{n}$ is $\beta = 0$.
\end{enumerate}
\end{thm}

\begin{proof}
The proofs for each of the above results are all quite similar, so details of the proofs for \Cref{it:ArvEq} and \Cref{it:EucEq} will be omitted.
The assertion given in \Cref{it:ArvEq} regarding Arveson extreme points is the content of \cite[Lemma 2.1 (3)]{EH19}. 
The assertion given in \Cref{it:EucEq} regarding Euclidean extreme points is the content of \cite[Corollary 2.3]{EHKM18}. It follows from \cite[Corollary 3]{RG95}.

To show \Cref{it:MatEq}, we note that by \Cref{theorem:Kriel} and \Cref{lemma:Kriel2}, $X$ is matrix extreme if and only if for $(\beta_0,\beta),(\gamma_0,\gamma)\in\SM{g+1}{n}$
\begin{align*}
    (Y_0,Y) = \begin{pmatrix} (I,X) & (\beta_0,\beta) \\ (\beta_0,\beta) & (\gamma_0,\gamma) \end{pmatrix} \in \H_{(I,A)}
\end{align*}
implies $(\beta_0,\beta) \in \vspan((I,X))$. 

Clearly $\ker \Lambda_{(I,A)}((I,X)) = \ker \LA(X)$ so $K_{A,X}$ is a matrix whose columns form an orthonormal basis of $\ker \Lambda_{(I,A)}((I,X))$ as well. If there is a $(\beta_0,\beta)$ such that $(Y_0,Y)\in\H_{(I,A)}$ and $(\beta_0, \beta)$ is not in the span of $(I,X)$, then by \Cref{eq:SchurGamma} it follows that
\begin{align*}
    \ker \Lambda_{(I,A)} ((I,X)) \subseteq \ker \Lambda_{(I,A)} ((\beta_0,\beta)).
\end{align*}
We can write $(\beta_0,\beta) = (\beta_0',\beta') + \alpha (I,X)$ for $\alpha\in\R$ and $(\beta_0',\beta') \neq 0$ satisfying \Cref{eq:matPerp}. The containment
\begin{align*}
    \ker \Lambda_{(I,A)} ((I,X)) \subseteq \ker \Lambda_{(I,A)} ((\beta_0,\beta)) = \ker (\Lambda_{(I,A)} (\alpha(I,X)) + \Lambda_{(I,A)} ((\beta_0',\beta')))
\end{align*}
implies that
\begin{align*}
    \ker \Lambda_{(I,A)} ((I,X)) \subseteq \ker \Lambda_{(I,A)} ((\beta_0',\beta')).
\end{align*}
Hence, there is a nonzero $(\beta_0',\beta')$ satisfying \Cref{eq:mat} and \Cref{eq:matPerp}.

Conversely, if there is $(\beta_0,\beta)$ satisfying \Cref{eq:mat} and \Cref{eq:matPerp}, then by taking $(\gamma_0,\gamma) = (I,0)$ the argument above reverses to show $X$ is not matrix extreme.
\end{proof}

\begin{cor} \label{cor:BeanCounts}

Let $A\in \SM{g}{d}$ be a minimal defining tuple for $\D_A$ and let $X \in \D_A(n)$. Recall that $k_{A,X} = \dim \ker \LA(X)$. Then,

\begin{enumerate}[\rm(1)]
    \item
    if $X$ is Arveson extreme then
    \begin{align} \label{eq:arvBeanCount} \index{Arveson Extreme Equation Count}
        gn \leq dk_{A,X};
    \end{align}
    
    \item
    if $X$ is Euclidean extreme then
    \begin{align} \label{eq:eucBeanCount} \index{Euclidean Extreme Equation Count}
        \frac{g(n+1)}{2} \leq dk_{A,X};
    \end{align}
    
    \item \label{it:MatBeanCount}
    if $X$ is matrix extreme then
    \begin{align} \label{eq:matBeanCount} \index{Matrix Extreme Equation Count}
        \frac{n(n+1)(g+1)}{2} \leq dnk_{A,X} + 1.
    \end{align}
\end{enumerate}
\end{cor}

For brevity, we will often refer to
\ben
\item $\left\lceil \frac{gn}{d} \right\rceil$ as the 
\arvcnt;
\item $\left\lceil \frac{g(n+1)}{2d} \right\rceil$ as the 
\euccnt;
\item 
$\left\lceil \frac{(n+1)(g+1)}{2d} - \frac{1}{nd} \right\rceil$ as the 
\extcnt.
\een
In this terminology, for $X$ to be a certain type of extreme point of $\cD_A$, the kernel dimension $k_{A,X}$ must be at least as large as the corresponding rank-nullity count.

\begin{proof}
The above inequalities result from comparing number of equations to number of unknowns in the homogeneous linear equations given in \Cref{theorem:Equations}. As an example the proof for \Cref{it:MatBeanCount} will be given below.

\Cref{eq:mat} and \Cref{eq:matPerp} are a set of homogeneous linear equations in the unknown $\beta \in\SM{g+1}{n}$. The matrix $\beta$ has $n(n+1)(g+1)/2$ scalar unknowns so \Cref{eq:mat} and \Cref{eq:matPerp} can be written in the form $M \beta' = 0$ where
\begin{align*}
    \beta &= \begin{pmatrix}
        \begin{pmatrix}
            \beta_{011} & \beta_{012} & \cdots & \beta_{01n} \\
            \beta_{012} & \beta_{022} & \cdots & \beta_{02n} \\
            \vdots & \vdots & \ddots & \vdots \\
            \beta_{01n} & \beta_{02n} & \cdots & \beta_{0nn}
        \end{pmatrix}
        ,
        \begin{pmatrix}
            \beta_{111} & \beta_{112} & \cdots & \beta_{11n} \\
            \beta_{112} & \beta_{122} & \cdots & \beta_{12n} \\
            \vdots & \vdots & \ddots & \vdots \\
            \beta_{11n} & \beta_{12n} & \cdots & \beta_{1nn}
        \end{pmatrix}
        ,
        \dots
        ,
        \begin{pmatrix}
            \beta_{g11} & \beta_{g12} & \cdots & \beta_{g1n} \\
            \beta_{g12} & \beta_{g22} & \cdots & \beta_{g2n} \\
            \vdots & \vdots & \ddots & \vdots \\
            \beta_{g1n} & \beta_{g2n} & \cdots & \beta_{gnn}
        \end{pmatrix}
    \end{pmatrix} \\
    \beta' &= (\beta'_0,
        \beta'_1,
        \cdots,
        \beta'_n
    ) \in \R^\frac{n(n+1)(g+1)}{2} \\ 
    \beta'_i &= (\beta_{i11}, \beta_{i12}, \cdots, \beta_{i1n}, \beta_{i22}, \beta_{i23}, \cdots, \beta_{i2n}, \beta_{i33}, \cdots, \beta_{inn})
\end{align*}
and $M\in M_{dnk_{A,X} + 1 \times \frac{n(n+1)(g+1)}{2}}(\R)$ (\Cref{eq:mat} has $dnk_{A,X}$ equations and \Cref{eq:matPerp} is one). 
 By the rank-nullity theorem, if \Cref{eq:matBeanCount} does not hold, then $M$ has a nontrivial nullspace and thus there is a nontrivial $\beta$ solving \Cref{eq:mat} and \Cref{eq:matPerp}. Hence by \Cref{theorem:Equations}, $X$ is not matrix extreme. 
\end{proof}

\sssec{The Dilation Subspace}
The space of all solutions, $\beta$, to \Cref{eq:arv} is often useful to consider. This is a subspace of $M_{n\times 1}(\R)^g$ and shall be referred to as the \df{dilation subspace} of $X$ with respect to $\D_A$. We call the dimension of this space the \df{dilation subspace dimension} of $X$ with respect to $\D_A$, often abbreviated $dilDim$.
The reference to $\D_A$ is often dropped when context makes it clear.

\sec{Matrix extreme points are always free extreme when $d=2$}
\label{sec:d=2}
We now examine free and matrix extreme points in the special case $d=2$. Our main result in this section is the following.

\begin{theorem}
\label{thm:d=2}
Let $A \in SM_2 (\R)^g$ and let $X \in \cD_A$. Then $X$ is a matrix extreme point of $\cD_A$ if and only if $X$ is a free extreme point of $\cD_A$. 
\end{theorem}
\begin{proof}
See Section \ref{ssec:d=2proof}.
\end{proof}

To prove Theorem \ref{thm:d=2}, we separately consider the cases $g=2$ and $g \geq 3$. The $g=2$ case is handled using projective maps to show that any bounded free spectrahedron $\cD_A$ with $A \in SM_2 (\R^2)$ can be mapped to a canonical free spectrahedron, see Proposition \ref{prop:g2d2CanForm}. Furthermore we show in the upcoming Theorem \ref{thm:MatExtToMatExt} that matrix extreme points are preserved under invertible projective transformations. The $g \geq 3$ case is handled by the following proposition which fully classifies all matrix and free extreme points of free spectrahedra that satisfy $g \geq d(d+1)/2$.

\begin{prop}
\label{prop:gverybig}
Let $A \in SM_d(\R)^g$. If $g\geq d(d+1)/2$, then $\cD_A$ is not bounded and is not the matrix convex hull of its free extreme points. Furthermore, such free spectrahedra either have exactly one matrix extreme point which is also free extreme in which case $g=d(d+1)/2$ or they have no extreme points at all. 

\end{prop}

\begin{proof}
To begin the proof, note that $\dim (SM_d(\R)) = d(d+1)/2$. Now if $g > d(d+1)/2$. then there exist constants $\alpha_1,\dots,\alpha_g \in \R$ such that $\sum_{i=1}^g \alpha_i A_i \otimes I=0$. As an immediate consequence, for any $X \in \cD_A$ and any $c \in \R$ we have $X+c(\alpha_1 I,\dots, \alpha_g I) \in \cD_A$. We conclude that $\cD_A$ has no extreme points at all if $g > d(d+1)/2$.

Now suppose $g = d(d+1)/2$. If $\{A_1,\dots,A_g\}$ is a linearly dependent set, then following the above argument shows $\cD_A$ has no extreme points. On the other hand, if $\{A_1,\dots,A_g\}$ is a linearly independent set, then there exist constants $\alpha_1,\dots,\alpha_g$ such that $\sum_{i=1}^g \alpha_i A_i = -I_2.$ We will show that $(\alpha_1,\dots,\alpha_g) \in \cD_A(1)$ is the only matrix extreme point of $\cD_A$. To this end let $X \in \cD_A(n)$ and observe that for any real number $c \geq 0$ we have
\[
L_A ((1-c)(\alpha_1 I,\dots,\alpha_g I)+cX) = I+ (1-c) I \otimes -I + c \Lambda(X) = c L_A (X) \succeq 0.
\]
It follows that $(1-c)(\alpha_1 I,\dots,\alpha_g I)+cX \in \cD_A$ for all $c \geq 0$ hence $X$  is not a Euclidean extreme point of $\cD_A$ unless $X = (\alpha_1 I,\dots,\alpha_g I)$. Moreover, since matrix extreme points are irreducible tuples, $X$ is not a matrix extreme point of $\cD_A$ unless $X = (\alpha_1,\dots,\alpha_g) \in \cD_A (1)$. It then follows from \cite[Proposition 6.1]{EHKM18} that this tuple is both a matrix and free extreme point of $\cD_A$. Finally, since $\cD_A$ is unbounded but only has one free extreme point, it is clear that $\cD_A$ cannot be the matrix convex hull of its free extreme points. 
\end{proof}

\ssec{Projective maps of free spectrahedra}

The proof of the $g=2$ case of Theorem \ref{thm:d=2} makes heavy use of projective mappings of free spectrahedra. We introduce here the basic definitions and notation related to projective maps. The definitions we use are the same as those found in \cite{E21}, and we direct the reader there for a detailed discussion of projective maps. 

To define projective maps of free spectrahedra, we first define linear mappings between homogeneous free spectrahedra. Given a matrix $W \in M_{g+1} (\R)$, we define a \df{linear transformation} $\cT_W$ on $SM(\R)^{g+1}$ by
\[
\cT_W (X_0,X)=\Big(\sum_{j=1}^{g+1} W_{1,j} X_{j-1} , \dots,\sum_{j=1}^{g+1} W_{g+1,j} X_{j-1} \Big) \qquad \mathrm{for \ all\ } (X_0,X) \in SM(\R)^{g+1}.
\]
Here $W_{i,j}$ denotes the $(i,j)$th entry of $W$. Note that $\cT_W$ is in fact a \df{free linear transformation} meaning that it respects direct sums and simultaneous unitary conjugation. Given a $g+1$ tuple $(A_0,A) \in SM_d(\R)^{g+1}$, the image of $\cH_{(A_0,A)}$ under $\cT_W$ is the set
\[
\cT_W (\cH_{(A_0,A)}) = \{ \cT_W (X_0,X) | \ (X_0,X) \in \cH_{(A_0,A)} \}.
\]
As one might expect, a linear transformation of a homogeneous free spectrahedron is again a homogeneous free spectrahedron, see \cite[Lemma 3.1]{E21}. If $\cH_{(A_0,A)}$ and $\cT_W (\cH_{(A_0,A)})$ are both positive free spectrahedra, then we say that $\cT_W$ is a \df{positive linear transformation} of $\cH_{(A_0,A)}$.

For a tuple $A \in SM_d(\R)^g$ and a matrix $W \in M_{g+1}(\R)$ such that $\cT_W$ is a positive linear transformation of $\fH(\cD_A)$, we define the projective transformation $\cP_W$ of $\cD_A$ by
\[
\cP_W(\cD_A) = \fH^{-1} (\cT_W (\fH (\cD_A))).
\]
In the case that $\cP_W (\cD_A)$ is bounded, we define $\cP_W$ on tuples $X \in \cD_A$ by
\[
\cP_W (X):= \fH^{-1} (\cT_W (\fH (X))
\]
Note that this is well-defined as a consequence of \cite[Lemma 2.2]{E21} since we have required that $\cT_W (\fH(\cD_A))$ is a sectionally bounded positive homogeneous free spectrahedron. In particular, \cite[Lemma 2.2]{E21} guarantees that the inhomogeneous component of $\cT_W (\fH(X))$ is positive semidefinite, hence $\fH^{-1}$ is well-defined on this tuple.  

This completes the definitions we require to proceed with the proof of Theorem \ref{thm:d=2}. Before doing so, we warn the reader that there are several idiosyncrasies that can occur when working with projective maps in the noncommutative setting. For example, it is necessary to insist that one uses a minimal defining tuple when defining the homogenization of a free spectrahedron, as otherwise the homogenization can fail to be well-defined. In addition, given a free spectrahedron $\cD_A$ and a projective map $\cP_W$ defined on $\cD_A$ where $W$ is invertible, while one always has the equality
\[
\cP_W (\cD_A) = \overline{\{ \cP_W (W) | X \in \cD_A \}},
\]
if $\cD_A$ is unbounded, then it can be the the case that 
\[
\cP_W (\cD_A) \neq \{ \cP_W (W) | X \in \cD_A \}.
\]
In fact, this issue necessarily occurs if $\cD_A$ is unbounded and $\cP_W (\cD_A)$ bounded. This for example has the consequence that $\cP_W$ can fail to be invertible even if $W$ is invertible. See \cite[Remark 2.3, Example 3.5]{E21} for further discussion. We lastly mention that while these idiosyncrasies can in principle occur, since we will always consider projective maps between bounded free spectrahedra, we will not encounter these issues, see \cite[Lemma 3.4]{E21}.

We now show that matrix extreme points are preserved under invertible projective transformations.

\begin{theorem}
\label{thm:MatExtToMatExt}
Let $\cD_A$ and $\cD_B$ be bounded free spectrahedra and suppose there exists some invertible matrix $W$ such that $\cP_W$ is an invertible projective map from $\cD_A$ to $\cD_B$. Then $X \in \cD_A$ is a matrix extreme point of $\cD_A$ if an only if $\cP_W (X)$ is a matrix extreme point of $\cD_B$.
\end{theorem}
\begin{proof}
Using Theorem \ref{theorem:Kriel} we have that $X$ is a matrix extreme point of $\cD_A$ if and only if the only solutions to 
\[
\ker \Lambda_{(I,A)} (I,X) \subset \ker \Lambda_{(I,A)} (Y_0,Y)
\]
satisfy $(Y_0,Y) = \alpha (I,X)$ for some $\alpha \in R$. 

Arguing by contrapositive, suppose $X$ is not a matrix extreme point of $\cD_A$. Then there is some tuple $(Y_0,Y)$ which satisfies
\[
\ker \Lambda_{(I,A)} (I,X) \subset \ker \Lambda_{(I,A)} (Y_0,Y)
\]
and that there is no $\alpha \in R$ such that $(Y_0,Y) = \alpha (I,X)$. From this we find
\[
\ker \Lambda_{\cT_{W^{-T}} (I,A)} \left(\cT_W (I,X)\right) \subset \ker \Lambda_{\cT_{W^{-T}} (I,A)} \left(\cT_W (Y_0,Y)\right).
\]
Furthermore, we cannot have $\alpha \cT_W (I,X) = \cT_W (Y_0,Y)$ since $\cT_W$ is an invertible linear transformation. As a consequence of \cite[Lemma 3.3]{E21} we have $\cT_{W^{-T}} (I,A)= (I,B)$ from which it follows that $\cT_W (I,X)$ is not on an extreme ray of $\cH_{(I,B)}$. Moreover, since $\cD_B$ is bounded by assumption, \cite[Lemma 2.2]{E21} shows that the inhomogeneous component of $\cT_W (I,X)$ is positive semidefinite. It is then straightforward to show that $(I,X)$ is not on an extreme ray of $\cH_{(I,B)}.$ We conclude that $\cP_W (X)$ is not a matrix extreme point of $\cD_B$, as claimed. 
\end{proof}

Knowing that matrix extreme points are preserved under invertible projective transformations, to treat the bounded $d=g=2$ case in Theorem \ref{thm:d=2}, it is sufficient to consider some canonical free spectrahedron which any other bounded $g=d=2$ free spectrahedron can be projectively mapped onto. The canonical free spectrahedron we consider is the spin disk since its matrix and free extreme points are well understood, see \cite[Proposition 7.5]{EHKM18}.

\begin{prop}
\label{prop:g2d2CanForm}
Let $A \in SM_2 (\R)^2$ and assume that the free spectrahedron $\cD_A$ is bounded and let $\cD_B$ be the spin disk. 
That is, $\cD_B$ is the free spectrahedron with defining tuple 
\[
 B =  \left(\begin{pmatrix}
1 & 0 \\
0 & -1 
\end{pmatrix},
 \begin{pmatrix}
0 & 1 \\
1 & 0
\end{pmatrix}\right).
\]
Then there exists an invertible matrix $W$ such that the map $X \mapsto \cP_W (X)$ defined on $\cD_A$ is a well-defined invertible projective transformation which maps $\cD_A$ onto $\cD_B$. 
\end{prop}

\begin{proof}

Write $A=(A_1,A_2)$ where
\[
A_1 = \begin{pmatrix}
a_{111} & a_{121} \\
a_{121} & a_{221}
\end{pmatrix} \qquad \qquad
A_2 = \begin{pmatrix}
a_{112} & a_{122} \\
a_{122} & a_{222}
\end{pmatrix}.
\]
Define the matrix $W \in \R^{3 \times 3}$ by 
\[
W = \begin{pmatrix}
1 & \frac{a_{111}+a_{221}}{2} & \frac{a_{112}+a_{222}}{2} \\
0 & \frac{a_{111}-a_{221}}{2} & \frac{a_{112}-a_{222}}{2} \\
0 & a_{121} & a_{122} \\
\end{pmatrix}.
\]
We will show that $\cP_{W}$ is a well-defined invertible projective map from $\cD_A$ onto $\cD_B$ and that $\cP_{W}$ has inverse $\cP_{W^{-1}}$. To accomplish this we must show that $W$ is invertible and that $\cT_W$ is a positive linear transformation that maps $\cH_{(I,A)}$ onto $\cH_{(I,B)}$. 

We first show that $W$ is invertible. To this end note that if $a_{122}$ and $a_{121}$ are both equal to zero, then $\cD_A$ is unbounded which is a contradiction.  Next note that that the determinant of $W$ is given by
\[
\det (W) = (a_{111} a_{122} - a_{121} a_{112} +a_{121} a_{222} - a_{221} a_{122})/2
\]
Thus, if $W$ is not invertible we have
\[
a_{111} a_{122} - a_{121} a_{112} = a_{221} a_{122} -a_{121} a_{222},
\]
from which it follows that the matrix
\[
a_{122} A_1 - a_{121} A_2 = \begin{pmatrix}
a_{111} a_{122} - a_{121} a_{112} & 0 \\
0 & a_{221} a_{122} -a_{121} a_{222}
\end{pmatrix}
\]
is either positive or negative semidefinite. In either case, $\cD_A$ is not bounded since the nonzero vector $\alpha (a_{122},-a_{121})$ is then an element of $\cD_A$ either for all $\alpha \geq 0$ or $\alpha \leq0$, depending on whether the quantity $a_{111} a_{122} - a_{121} a_{112}$ is positive or negative. We conclude that $W$ is invertible. 

Next observe that a direct calculation shows that $\cT_{W^T} (I,B_1,B_2) = (I,A_1,A_2)$ hence $\cT_{W^{-T}} (I,A_1,A_2) = (I,B_1,B_2).$ Using \cite[Lemma 3.1]{E21} then shows that 
\[
\cT_W (\cH_{(I,A)}) = \cH_{(I,B)}.
\] 
Since $\cH_{(I,A)}$ and $\cH_{(I,B)}$ are both positive homogeneous free spectrahedra, we obtain that $\cT_W$ is a positive linear transformation from $\cH_{(I,A)}$ to $\cH_{(I,B)}$. Furthermore, $\cH_{(I,B)}$ is sectionally bounded since $\cD_B$ is bounded. It follows that the map $X \mapsto \cP_W (X)$ defined on $\cD_A$ is indeed a well-defined projective transformation that maps $\cD_A$ into $\cD_B$. The proof is completed by \cite[Lemma 3.3]{E21} and \cite[Lemma 3.4]{E21} which together show that $\cP_W$ maps $\cD_A$ onto $\cD_B$ and that this map is invertible with inverse equal to $\cP_{W^{-1}}$.
\end{proof}

\ssec{Proof of Theorem \ref{thm:d=2}}
\label{ssec:d=2proof}


The result in the case that $g \geq 3$ if proved as Proposition \ref{prop:gverybig}, so it is sufficient to consider $g \leq 2$.  We first assume that $g=2$ and that $\cD_A$ is unbounded.  Equivalently, assume that $\mathrm{span} (\{A_1,A_2\})$ contains a positive semidefinite matrix. We first argue that it is sufficient to consider tuples $(A_1,A_2)$ of the form
\[
A_1 = \begin{pmatrix} 1 & 0 \\ 0 & 0 \end{pmatrix} \qquad \qquad A_2 = \begin{pmatrix}
0 & 1 \\ 1 & b
\end{pmatrix}
\]
where $b \in \R$. To this end, note that if $c I \in \mathrm{span} (\{A_1,A_2\})$ for some $0 \neq c \in \R$ or if $\{A_1,A_2\}$ is linearly dependent, then we can repeat the argument used in Proposition \ref{prop:gverybig} to conclude that all matrix extreme points of $\cD_A$ are also free extreme. 
Therefore we assume that $c I \notin \mathrm{span} (\{A_1,A_2\})$ and that $\{A_1,A_2\}$ is linearly independent. In this case, a routine argument shows that $\mathrm{span} (\{A_1,A_2\})$ contains a rank one positive semidefinite matrix. 
Furthermore, it is straightforward to show that if there exists an invertible linear transformation on $SM(\R)^2$ which maps $A$ to $B$, then $\cD_A$ contains a matrix extreme point which is not free extreme if and only if $\cD_B$ contains a matrix extreme point that is not free extreme. Using this fact, we can without loss of generality assume that $A_1$ has rank $1$ and that the spectrum of $A_1$ is $\{1,0\}.$ From here we can use the fact that unitarily equivalent tuples define the same free spectrahedron together with another invertible change of a variables to reduce to the case 
\[
A_1 = \begin{pmatrix} 1 & 0 \\ 0 & 0 \end{pmatrix} \qquad \qquad A_2 = \begin{pmatrix}
0 & 1 \\ 1 & b
\end{pmatrix}
\]
as claimed. 

Now, with $(A_1,A_2)$ as above, we can use the Schur complement to conclude that $X \in \cD_A$ if and only if
\[
I+b X_2 \succeq 0 \qquad \qquad \mathrm{and} \qquad \qquad I+X_1 - X_2 (1+bX_2)^\dagger X_2 = P.
\]
Observe that if $X \in \cD_A$ then 
\[
I+(X_1+P)-X_2 (1+b X_2)^\dagger X_2 = 2P \succeq 0 \qquad \mathrm{and} \qquad I+(X_1-P)-X_2 (1+b X_2)^\dagger X_2 = 0
\]
from which we would obtain that $(X_1+P,X_2) \in \cD_A$ and $(X_1-P,X_2) \in \cD_A$. Moreover, if $P \neq 0$, then 
\[
(X_1+P,X_2) \neq X \neq (X_1-P,X_2).
\] 
That is, if $X$ is a Euclidean extreme point of $\cD_A$, then $I+X_1 - X_2 (1+bX_2)^\dagger X_2 = 0$. 

However in this case we have 
\[
X_1 X_2 = (X_2 (1+bX_2)^\dagger X_2-I) X_2 = X_2 (X_2 (1+bX_2)^\dagger X_2-I) = X_2 X_1.
\]
We conclude that if $X \in \cD_A(n)$ is a Euclidean extreme point of $\cD_A$ and $n \geq 1$, then $X$ is reducible. It follows from \cite[Theorem 1.1]{EHKM18} that all matrix and free extreme points of $\cD_A$ are contained in $\cD_A(1)$. From here, one can use \cite[Proposition 6.1]{EHKM18} to show that all matrix extreme points of $\cD_A$ are also free extreme points. 

Now suppose that $g=2$ and that $\cD_A$ is bounded. Using Proposition \ref{prop:g2d2CanForm} shows that there exists some invertible projective transformation $\cP_W$ which maps $\cD_A$ onto $\cD_B$ where
\[
B = \left( \begin{pmatrix}
1 & 0 \\
0 & -1
\end{pmatrix},
\begin{pmatrix}
0 & 1 \\
1 & 0
\end{pmatrix} \right).
\]
Furthermore using Theorem \ref{thm:MatExtToMatExt} shows that a tuple $X$ is a matrix extreme point of $\cD_A$ if and only if $\cP_W (X)$ is a matrix extreme point of $\cD_B$. Similarly, \cite[Theorem 3.7]{E21} show that $X$ is a free extreme point of $\cD_A$ if and only if $\cP_W (X)$ is a free extreme point of $\cD_B$. Combining this with \cite[Proposition 7.5]{EHKM18} which shows that every matrix extreme point of $\cD_B$ is also free extreme completes the proof in this case.

Finally, for completeness, we mention that the proof when $g=1$ is straightforward. In this case it is easy to show that all matrix extreme points of $\cD_A$ are found at level $1$ of $\cD_A$, hence all matrix extreme points are free extreme.
\hfill\qedsymbol

\sec{Exact \mnae Points} \label{section:exact-points}

In this section we present examples of 
\mnae points for free spectrahedra when $g=3$ and $g=4$. For $g=2$ the 
existence of \mnae{} points is not known.
To prove an example has our claimed properties, numerical (floating point) calculations do not suffice due to possible  numerical errors. However, usually it is difficult to find extreme points of free spectrahedra with exact arithmetic. Even in the smallest nontrivial case where the defining tuple $A\in \SM{2}{2}$ and the desired extreme point $X\in \D_A(2),$ exactly computing $X$ by optimizing a linear functional requires computing an exact arithmetic solution to a semidefinite program with six variables and the constraint $\LA(X) \psd 0$ where $\LA(X)$ is $4\times 4$.

If the size of the defining tuple or the extreme point is greater than two, then computing even just a boundary point would require finding an exact solution $\alpha$ for $\det(\LA(\alpha X)) = 0$, which is often not possible in radicals since $\det(\LA(\alpha X))$ is generally a polynomial of degree greater than five in $\alpha$. As a consequence, it can in some cases be impossible to express boundary points, let alone extreme points, using radicals.
Algorithms which sometimes yield exact extreme points  (typically expressed using  roots of some polynomial)  are  the subject of \Cref{section:igors-method}.

\ssec{$g=3$ \mnae example}
\label{sec:exactg3ex}

Now we give an example for $g=3$ 
of a bounded free spectrahedron $\cD_A$
and a \mnae point $Y$ in it.
In this case, $A$ and  $Y$ have entries which are algebraic numbers 
and we have proved using exact arithmetic that $Y$
is in $\cD_A$ and is not Arveson extreme.
To  prove that $Y$ is not matrix extreme one only needs to check that the 
Matrix Extreme \Cref{eq:mat} has no solution. This we proved via floating point arithmetic
by checking that the appropriate matrix has 
smallest singular value equal to $0.0318244$ while largest singular value is not very large ($<5$),
hence it has no nullspace.
Alternately, one could prove this using exact arithmetic, 
which would be considerably slower in this case.

Now we state our example. Let
\begin{align*}
    A = \left(
    \left(
\begin{array}{cccc}
 0 & 0 & -1 & 1 \\
 0 & 0 & 1 & 0 \\
 -1 & 1 & 0 & 1 \\
 1 & 0 & 1 & 1 \\
\end{array}
\right),\left(
\begin{array}{cccc}
 -1 & -1 & 1 & 1 \\
 -1 & 0 & 0 & 1 \\
 1 & 0 & -1 & -1 \\
 1 & 1 & -1 & 0 \\
\end{array}
\right),\left(
\begin{array}{cccc}
 -1 & 0 & 0 & 0 \\
 0 & 0 & 0 & 0 \\
 0 & 0 & 1 & 1 \\
 0 & 0 & 1 & 1 \\
\end{array}
\right)
    \right)
\end{align*}
be the defining tuple of the free spectrahedron $\D_A$ and
\begin{align*}
    Y = \left(
    \left(
\begin{array}{ccc}
 \frac{1}{4} & \frac{27}{100} & \alpha \\
 \frac{27}{100} & -\frac{13}{100} & \alpha \\
 \alpha & \alpha & 0 \\
\end{array}
\right),
\left(
\begin{array}{ccc}
 -\frac{27}{100} & \frac{21}{100} & 3 \alpha \\
 \frac{21}{100} & \frac{7}{100} & \alpha \\
 3 \alpha & \alpha & 0 \\
\end{array}
\right),
\left(
\begin{array}{ccc}
 \frac{7}{50} & -\frac{49}{100} & 3 \alpha \\
 -\frac{49}{100} & \frac{3}{10} & 0 \\
 3 \alpha & 0 & 0 \\
\end{array}
\right)
\right)
\end{align*}
where $\alpha$ is the smallest positive root of the polynomial 
\begin{align*}
p(t) & = 20828330523 - 3649588559100t^2 + 132250437590000t^4 \\
& \phantom{{}={}} - 
 651404153000000t^6 + 748026200000000t^8.
\end{align*}
Then $Y$ is a \mnae  point of $\D_A$.
We know that $p$ has a positive real zero as $p(0) = 20828330523$ and $p(\frac{1}{8}) = -\frac{208047637414661}{32768}.$ So by the intermediate value theorem, there must be a zero between $0$ and $\frac{1}{8}$. We go into detail on how this point was computed and proved to be a \mnae point of $\D_A$ in \Cref{section:igors-method}.

\ssec{Exact arithmetic \mnae  point for $g=4$}

For $g=4$, we also have an exact arithmetic example of a \mnae  point. 
Let $A = (A_1,A_2,A_3,A_4)$ for
\begin{align*}
    &A_1 = \mathrm{diag}\left(2,0,-4,0,0,0,-4,0,\frac{8}{3}\right), & A_2 = \mathrm{diag}\left(0,4,-4,0,0,0,0,-\frac{8}{3},\frac{8}{3}\right) \\
    &A_3 = \mathrm{diag}\left(0,0,0,4,0,-\frac{8}{3},-4,0,\frac{8}{3}\right), & A_4 = \mathrm{diag}\left(0,0,0,0,\frac{8}{3},-\frac{8}{3},0,-\frac{8}{3},\frac{8}{3}\right)
\end{align*}
where $\mathrm{diag}(v)$ is the diagonal matrix whose diagonal is the vector $v$, be the defining tuple of the free spectrahedron $\D_A$. Then the tuple $X = (X_1,X_2,X_3,X_4)$ for
\begin{align*}
    X_1 &= \left(
\begin{array}{cc}
 -\frac{1}{2} & 0 \\
 0 & \frac{3}{10} 
\end{array}
\right)\qquad
    X_2 =  \left(
\begin{array}{cc}
 \frac{1}{2} & \frac{\sqrt{\frac{3}{5}}}{4} \\
 \frac{\sqrt{\frac{3}{5}}}{4} & -\frac{1}{5} \\
\end{array}
\right) \\
    X_3 &= \left(
\begin{array}{cc}
 \frac{1521520 \sqrt{3}-619599 \sqrt{182}}{1019200 \sqrt{3}-1197204 \sqrt{182}} & 0 \\
 0 & -\frac{1}{4} \\
\end{array}
\right) \\
    X_4 &= \left(
\begin{array}{cc}
 \frac{5 \left(1664 \sqrt{546}-124455\right)}{3143688} & -\frac{4 \left(1820 \sqrt{15}+669 \sqrt{910}\right)}{392961} \\
 -\frac{4 \left(1820 \sqrt{15}+669 \sqrt{910}\right)}{392961} & \frac{11200 \sqrt{546}-429603}{3143688} \\
\end{array}
\right)
\end{align*}
is a \mnae  point of $\D_A$. 

The tuple $X$ was computed by first taking an interior point of $\D_A$ with rational entries, and then perturbing each entry in order increase the kernel dimension $k_{A,X}$ to a suitable size. We can see the result of this method in the structure of $X$, as each of the $X_i$'s is progressively more complicated.

Of note in this example is that the defining tuple $A$ is a tuple of diagonal matrices, and thus $\D_A$ is a free polytope. It is thanks to this that we can use exact arithmetic to verify
that the tuple $X = (X_1,X_2,X_3,X_4)$
 is \mnae. 

\sec{Exact Arithmetic Point Generation} \label{section:igors-method}

This section concerns methods for producing 
provable examples of \mnae  points $X$.

We look at two ``exact'' methods. The first has been effective at producing lots of examples $X$ when $g=3$, $d=4$, and $n = 3.$ The second though, run extensively for 
 $g=2$, $d=3$, and $n = 8$, failed to produce any exact examples, though it produced many examples which are numerically promising. 
Beyond these parameters,  we have not explored either algorithm since our implementations require some intervention, hence are not fast. 
The example of a \mnae  point in \Cref{section:exact-points} was produced using the first of these exact methods.

\ssec{An algorithm for finding exact arithmetic extreme points}

The algorithm described below generates a boundary point $X$ and then
dilates in a precise manner to a point
\[
Y = \begin{pmatrix}
X & \beta \\ \beta^T & 0
\end{pmatrix}
\]
such that $\dim \ker \LA(Y) = 2$.  Such a $Y$ is then a good candidate for a \mnae  point for $g=3,d=5,$ and $n = 3$, since the \arvcnt here is three, and the \extcnt is 2.

\begin{algo}
\label{algo:exact-dilation}
Let $A \in SM_d(\{-1,0,1\})^g$ such that $\D_A$ is a bounded real free spectrahedron and fix $n\in \N$. Pick a $K \in \{-1,0,1\}^{d n}$ uniformly at random and solve (rational arithmetic) the linear systems
\begin{equation}\label{eq:linearsystems}
\begin{split}
    \LA(X) K &= 0 \\
    \lA(\beta^T)K &= 0 
\end{split}
\end{equation}
for $X\in SM_n(\Q)^g$ and $\beta \in M_{n\times 1}(\Q)^g$. 

\begin{enumerate}[\rm(1)]
    \item If no solution exists for this $K$, choose a new $K \in \{-1,0,1\}^{d n}$ uniformly at random and solve (rational arithmetic) the linear systems \eqref{eq:linearsystems}.
   
    \item If one solution exists for this $K$, we must then check that $\LA(X)$ is positive semidefinite (floating point arithmetic). If $\LA(X)$ is not positive semidefinite, we discard this $K$ and choose a new $K \in \{-1,0,1\}^{d n}$ uniformly at random and solve (rational arithmetic) the linear systems \eqref{eq:linearsystems}.

    \item If there are infinitely many solutions, we then pick a tuple $X$ in the solution space such that $\LA(X)$ is positive semidefinite.

\item 
Once we have such an $X$ and $\beta$, we let
\begin{align*}
    \widehat{Y}(\widehat \alpha) = \begin{pmatrix}
        X & \widehat{\alpha}\beta \\
        \widehat{\alpha}\beta^T  & 0
    \end{pmatrix}
\end{align*}
where $\widehat{\alpha} \in \R$.
Let $\alpha$ denote the smallest non zero root  of the
derivative 
$p_1(\widehat \alpha) :=\frac {d \chi_{\widehat \alpha}(t) }{dt}_{|_{t=0}}$
of the 
characteristic polynomial $\chi_{\widehat \alpha}(t)$ of $\LA(\widehat{Y}(\widehat \alpha ))$.

\item Denote
$p_2(\widehat \alpha) :=\frac {d^2 \chi_{\widehat \alpha}(t) }{d^2t}_{|_{t=0}}$. If $p_2(\widehat \alpha) = 0$, then
we generate a new $X$ and $K$ and repeat the process. Otherwise return
\[
Y:= \widehat Y (\alpha).
\]
\een

\noindent The algorithm uses exact arithmetic so that $A$, $Y$, and $K$ have entries which are algebraic numbers.
\end{algo}

\begin{theorem} \label{theorem:igors-method}
If the above algorithm terminates,
then the point $Y$ it returns will belong to $\cD_A$ and have $k_{A,Y} = 2$. 
Hence, if $g=3$,
$d=4$, and $Y \in \SM{3}{3}$, we have
that the Arveson 
\Cref{eq:arv} has more unknowns than constraints,
so $Y$ is not an Arveson extreme point.
\end{theorem}

\begin{proof}

Let $X$, $\beta$, $K$, and $\alpha$ be given by the algorithm above and let
\begin{align*}
    Y = \begin{pmatrix}
        X & \alpha\beta \\
        \alpha\beta^T & 0
    \end{pmatrix}.
\end{align*}
We know that
\begin{align*}
    \LA(Y) &= \Pi^T\begin{pmatrix}
        \LA(X) & \alpha\lA(\beta) \\
        \alpha \lA(\beta^T) & I
    \end{pmatrix}\Pi\\
\end{align*}
for some unitary $\Pi$  (namely the canonical shuffle),
so letting
\begin{align*}
    K' = \Pi^T\begin{pmatrix}
        K \\
        0
    \end{pmatrix},
\end{align*}
we see that $\LA(Y) K' = 0$ as
\begin{align*}
    \LA(Y) &= \Pi^T
    \begin{pmatrix}
        \LA(X) & \alpha\lA(\beta) \\
        \alpha \lA(\beta^T) & I
    \end{pmatrix} \Pi K' \\
    &= \Pi^T\begin{pmatrix}
        \LA(X)K & 0 \\
        \alpha \lA(\beta^T) K & 0
    \end{pmatrix} = 0
\end{align*}
as $\LA(X)K = \lA(\beta^T) K = 0.$
Thus, the characteristic polynomial $\chi_\alpha(t)$ of $\LA(Y)$ has no constant term. Moreover, by the definition of $\alpha$, we have $\chi_\alpha'(0) = 0$ and $\chi_\alpha''(0) \neq 0$. Thus $\chi_\alpha(t) = t^2 q_\alpha(t)$ for some polynomial $q_\alpha$ such that $q_\alpha(0) \neq 0$, consequently $\LA(Y)$ has a nullspace of dimension $2$.

We now aim to show that $Y\in \D_A$. To do this, we first note that taking $\widehat{\alpha} = 0$, we get that $\LA(\widehat{Y}(0))$ is unitarily equivalent to
\[
\begin{pmatrix}
        \LA(X) & 0 \\
        0 & I
    \end{pmatrix}
\]
which is clearly positive semidefinite. 
Moreover, since $\LA(\widehat{Y}(\widehat{\alpha})) K' = 0$ for all $\widehat{\alpha}$, we have $\chi_{\widehat{\alpha}}(t) = t(t - \lambda_2(\widehat{\alpha}))\cdots(t - \lambda_{nd+d}(\widehat{\alpha}))$ where the $\lambda_i(\widehat{\alpha})$ are the eigenvalues of $\LA(\widehat{Y}(\widehat{\alpha}))$ that are not identically zero in $\widehat{\alpha}$. 

Thus, $p_1(\widehat{\alpha}) = (-1)^{(n+1)d - 1} \lambda_2(\widehat{\alpha}) \cdots \lambda_{nd+d}(\widehat{\alpha}).$ We note that $p_1(\widehat{\alpha}) = 0$ if and only if $\lambda_i(\widehat{\alpha}) = 0$ for some $i = 2,3,\dots,nd+d$ and we pick $\alpha$ to be the smallest positive root of $p_1(\widehat{\alpha}).$ Thus, there must be some $i$ such that $\lambda_i(\alpha) = 0$. Moreover, if $\lambda_j(\alpha) < 0$, then, by the intermediate value theorem since $\lambda_j(0) > 0$, there must exist some $0 < \alpha_0 < \alpha$ where $\lambda_j(\alpha_0) = 0$. This implies $p_1(\alpha_0) = 0$, a contradiction to the assumption that $\alpha$ is the smallest positive root of $p_1(\widehat{\alpha})$. Thus for $j \neq i$, $\lambda_j(\alpha) \geq 0$ and hence $Y\in \D_A$.

The \arvcnt (\ref{eq:arvBeanCount}) for $g=3$, $d=4$, and $n=3$ is $3.$ So by \Cref{cor:BeanCounts}, if $Y\in \D_A(3)$ then $k_{A,Y} = 2 < 3$ implies that $Y$ is not Arveson extreme.
\end{proof}

Thus, any point produced with this algorithm at these parameters cannot be Arveson extreme, but is potentially matrix extreme. The matrix $\LA(Y)$ will have entries that are algebraic numbers, and so the kernel and matrix extreme equations can, in principle, be computed in exact arithmetic.

\ssec{A property of the characteristic polynomial $\chi_\alpha$}

\begin{lemma}
If
\[
Y = \begin{pmatrix}
    X & \alpha\beta \\
    \alpha\beta^T & 0
\end{pmatrix}
\]
for $X\in\SM{g}{n}$, $\beta\in M_{n\times 1}$, and $\alpha\in \R$, then the characteristic polynomial $\chi_\alpha(t)$ of $\LA(Y)$ has coefficients that are degree $dn$ polynomials in $\alpha^2$.
\end{lemma}

\begin{proof}
Using the canonical shuffle, we can show that $\LA(Y)$ is unitarily equivalent to
\begin{align*}
    Z = \begin{pmatrix}
        \LA(X) & \alpha\lA(\beta) \\
        \alpha \lA(\beta^T) & I
    \end{pmatrix}.
\end{align*}
The eigenvalues of a matrix, and thus the characteristic polynomial, are invariant under unitary equivalence, so it is sufficient to compute the characteristic polynomial of $Z$. Note that $Z \in SM_{d(n+1)}(\R[\alpha])$ and thus $\chi_\alpha(t) \in \R[\alpha][t]$. 
\begin{align*}
    \chi_\alpha(t) = \det\begin{pmatrix}
        \LA(X) - tI & \alpha\lA(\beta) \\
        \alpha \lA(\beta^T) & (1-t)I
    \end{pmatrix}
\end{align*}
so for $t\neq 1$, we can use the Schur determinant formula to show that
\begin{align*}
    \det\begin{pmatrix}
        \LA(X) - tI & \alpha\lA(\beta) \\
        \alpha \lA(\beta^T) & (1-t)I
    \end{pmatrix} &= \det((1-t)I) \det([\LA(X) - tI - \frac{\alpha^2}{(1-t)} \lA(\beta) \lA(\beta^T)]) \\
    &= (1-t)^d\det([\LA(X) - tI - \frac{\alpha^2}{(1-t)} \lA(\beta) \lA(\beta^T)]).
\end{align*}
Thus $\chi_\alpha(t)$ depends only on $\alpha^2$ for $t\neq 1$. The matrix in the square brackets above is $dn \times dn$, so this determinant has degree $dn$ in $\alpha^2$. 
\end{proof}

\sssec{Experiments with   \Cref{algo:exact-dilation}}
We ran \Cref{algo:exact-dilation} one hundred times; every time an $X$, $\beta$, and $K$ were found. Each time the procedure was run, multiple kernels $K$ were generated until an $X$ and $\beta$ could be found such that $\LA(X) K = 0$ and $\lA(\beta^T)K = 0.$ We have only run this method for a single fixed defining tuple $A$.

On two randomly chosen occasions, we went through the hour long process of trying to compute the matrix extreme equations exactly. While the equations could in principle be computed in exact arithmetic, the calculation for determining if the matrix extreme equations had a nullspace became too difficult to solve.

We then numerically determined the singular values of the matrix extreme equations, \Cref{eq:mat}. In every case, the largest singular value was less than $4$ and $80$ of the points had smallest singular value on the order of $10^{-4}.$ Hence, we are confident that those $80$ points are matrix extreme. In order to be Arveson extreme, the points would require a kernel dimension of at least $3$. However, in all cases, we computed the eigenvalues of $\LA(X)$ numerically and determined two of the eigenvalues to be zero. This determination was made as the smallest two eigenvalues were on the order of $10^{-15}$ in all cases, and the next smallest eigenvalue was on the order of $10^0$. Thus, the points could have a kernel of at most dimension $2$.  Note here that Theorem \ref{theorem:igors-method} shows that $L_A(Y)$ has kernel dimension $2$, so the numerical results match our expectations. 

\begin{remark}
Nowhere in \Cref{theorem:igors-method} do we use the fact that $A,K$ were chosen from $\set{-1,0,1}.$ In fact, we may pick $A,K$ with rational entries. 
The choice to go with $\set{-1,0,1}$ entries was made solely
 to reduce the computational complexity of the experiments.
\end{remark}

\subsection{$g=2$, the Wild Disc}
\label{section:exactAlgg2}

For $g = 2$, it is still open whether there exist \mnae  points. 
This difficulty is demonstrated by an informal method 
for the specific example of the wild disc.
This method works well to produce floating point \mnae \ candidates but 
many
 attempts with exact arithmetic did not
 yield a provable example.
 %
 However, we strongly 
   strongly conjecture that \mnae  points exist for the wild disc.

The $g=2$ wild disc
is the free spectrahedron described by
\[
A_1=\begin{pmatrix}
0& 1 & 0 \\
1 & 0 & 0\\
0&0&0
\end{pmatrix},
\quad
A_2=\begin{pmatrix}
0& 0 & 1 \\
0 & 0 & 0\\
1&0&0
\end{pmatrix}.
\]
Thus
\[
L_A(X,Y)=\begin{pmatrix}
1& X & Y \\
X & 1 & 0\\
Y &0& 1
\end{pmatrix},
\]
and an easy Schur complement calculation shows that
\[
\cD_A=\{(X,Y) \mid S(X,Y):=I-X^2-Y^2\succeq0\}.
\]

\subsubsection{Finding many \mnae points}

\begin{lemma}
We have $(X,Y)\in \partial \cD_A$ iff
$S(X,Y)\succeq0$ is singular.
More precisely,
\[
\ker L_A(X,Y)= \Big\{ \begin{pmatrix} v \\ -Xv \\ -Y v \end{pmatrix} \mid
v \in \ker S(X,Y) \Big\}.
\]
\end{lemma}

\begin{proof}
Straightforward.
\end{proof}

Based on experimental conclusions described in
\Cref{section:g=2}, we expect that there are
pairs $(X,Y)$  of $8\times8$ matrices that are \mnae points of $\cD_A$. Indeed, one can attempt to produce many such examples by starting with a random $8\times8$ matrix
\[
0\preceq S\preceq I
\] 
whose  rank is  $8-3=5$. Then choose $X$ with
\[X^2\preceq I-S. \]
All this is easily done with 
exact arithmetic.

The difficult part is computing exact  
\beq 
\label{eq:xys}
Y=(I-X^2-S)^{1/2}.
\eeq 
In cases where it was easy to compute exact $Y$, we found
that $(X, Y)$ were not \mnae. 
The approaches we used to generate easy to compute $Y$ typically introduced reducibility or some other degeneracy, hence these points were not \mnae.
On the other hand 
with nonzero probability
(in informal experiments)
the pair $(X,Y)$ computed in floating point was a \mnae point.

Alternately,  given $S$ as above, to find $X,Y$
one could 
pick a $2\cdot 8\times 8$ isometry
\[
\begin{pmatrix} W_1\\W_2\end{pmatrix}
\]
(e.g.~by picking first 8 columns of a $16\times16$ random unitary matrix) and let
\[
\begin{pmatrix} X'\\Y' \end{pmatrix}:=
\begin{pmatrix} W_1\\W_2 \end{pmatrix} (I-S)^{1/2}.
\]
Of course, $X',Y'$ won't be self-adjoint. So we correct for this by
multiplying $W_i$ with a symmetry $U_i$ commuting with $|W_i(I-S)^{1/2}|$, where $|R|:=(R^TR)^{1/2}$. Then
\[
X= U_1 |W_1(I-S)^{1/2}|, \quad Y=U_2 |W_2(I-S)^{1/2}|
\]
are self-adjoint with $I-X^2-Y^2=S$.
Finding appropriate unitaries $U_1, U_2$
with exact arithmetic is formidable and pursuing this did not produce an exact \mnae  point.
However, as before
with nonzero probability (in our experiments) the floating point pair $(X,Y)$ was a \mnae
point.

\sec{Numerical Algorithms for Dilation to Extreme Points} 
\label{sec:algorithms}

Now we turn to numerical experiments; this section gives the underlying algorithms.
This paper will use these algorithms in two different ways.
One is to generate numerical candidates for \mnae  points;
this will be the main tool behind 
\Cref{sec:experiments}.
The second goal is a reliable and accurate algorithm for producing a free Caratheodory expansion of  a matrix tuple $Y$ inside
a free spectrahedron;
this is equivalent to dilating $Y$
to an Arveson extreme point, see \Cref{sec:FreeExtSignificance}.
Success rates and statistics on 
complication of the dilation  are
 in
 \Cref{sec:free-carath}. We mention that all of the algorithms described in this section have been implemented  and are publicly available in NCSE \cite{EOYH21}.

As motivation we elaborate on our \mnae  point objective.
While there exist points that are \mnae  for free spectrahedra in $\SM{3}{n}$ and $\SM{4}{n}$, as demonstrated in \Cref{section:exact-points}, there is still the question of how frequently these points occur. To answer this question, we turn from methods of generating exact arithmetic extreme points to generating extreme points numerically. Numerical testing is an effective way of generating and testing a large number of points relatively quickly in order to get an idea (without proof) of whether or not example points are rare.

The most obvious method of generating an extreme point of a free spectrahedron $\D_A$ would be to optimize a random linear functional $\ell(X)$ under the constraint $\LA(X) \psd 0$. However, in practice a surprisingly large majority of points generated by this method are Arveson extreme and have large $k_{A,X}$, as shown in \cite{EFHY21}. Intuitively, and in fact empirically, we expect that points with large $k_{A,X}$ are more likely to be Arveson extreme. This intuition is based on  \Cref{cor:BeanCounts}. As such, we wish to employ a method of generating extreme points that yields points with small kernels.\looseness=-1

\subsection{Extreme point generation via dilation} \label{subsection:extreme-point-generation}

We  begin by describing 
an algorithm  found in \cite{EFHY21}, which tries to dilate
a given $Y$ of size $n_0$ 
in a free spectrahedron $\cD_A$ 
to an 
Arveson extreme point.
This is done by computing a series of carefully chosen dilations starting with $Y$.

 We find that the 
  success  of this algorithm depends heavily on accuracy in computing $\ker \LA(Y^{j})$.
  Nullspace Purification, a method for achieving much improved accuracy, 
  is introduced and tested here
  in conjunction with 
  \Cref{algo:pure-dilation-algo}.
  Using Nullspace Purification results in  \Cref{algo:pure-dilation-algo} becoming  reliable for broad classes of problems.\looseness=-1
  
\ssec{The workhorse extremal dilation algorithm}

\begin{algo}[Extremal Dilation Algorithm \cite{EFHY21}]
\label{algo:pure-dilation-algo}
Let $A\in\SM{g}{d}$ be such that $\D_A$ is a bounded  free spectrahedron and let an initial point $Y^0 \in \D_A(n_0)$ be given. For integers $j=0,1,2,\dots$ such that $Y^j$ is not an Arveson extreme point of $\D_A$, define
\begin{align*}
    Y^{j+1}:=\begin{pmatrix}
        Y^j & c_j \widehat{\beta}^j \\
        c_j (\widehat{\beta}^j)^T & \widehat{\gamma}^j
    \end{pmatrix}
\end{align*}
where $\widehat{\beta}^j$ is a nonzero solution to
\begin{align*}
    \ker \LA(Y^j) \subseteq \ker \lA(\beta^T), \quad \beta \in M_{n\times 1}(\R)^g
\end{align*}
and where $c_j$ and $\widehat{\gamma}^j$ are solutions to the sequence of maximization problems
\begin{align*}
    c_j:=& \quad \underset{c\in\R,\gamma\in\R^g}{\mathrm{Maximizer}} \quad c \\
    \mathrm{s.t.}& \quad \LA \begin{pmatrix}
        Y^j & c \widehat{\beta}^j \\
        c (\widehat{\beta}^j)^T & \gamma
    \end{pmatrix} \psd 0 \\
    \mathrm{and} \quad \widehat{\gamma}^j :=& \quad \underset{\gamma\in\R^g}{\mathrm{Maximizer}} \quad \ell(\gamma) \\
    \mathrm{s.t.}& \quad \LA \begin{pmatrix}
        Y^j & c_j \widehat{\beta}^j \\
        c_j (\widehat{\beta}^j)^T & \gamma
    \end{pmatrix} \psd 0. 
\end{align*}
Here $\ell:\R^g\to\R$ is a random linear functional.

If $Y^j$ is an Arveson extreme point, we instead terminate the algorithm. 
\end{algo}

\ssec{Generating \mnae  points}
In these tests, we do not aim to generate Arveson extreme points, and in fact, we are trying to actively avoid generating Arveson extreme points, so an algorithm for generating such points may seem like a strange place to look. 

However, if $Y^0$ is the initial point, and $Y^j$ is the point generated by \Cref{algo:pure-dilation-algo}, then there are $j - 1$ points, $Y^1, Y^2, \dots, Y^{j - 1}$, that we have generated that are ``close'' to being Arveson extreme but, crucially, are not. As we describe in this section, these sets of points provide many examples of \mnae  points.

In the experiments detailed in \Cref{sec:experiments}, we will employ the following algorithm for numerical extreme point generation.

\begin{algo}
\label{algo:dilation-algorithm}
Let $A\in\SM{g}{d}$ such that $\D_A$ is a bounded real free spectrahedron. Given an interior point $X \in \D_A(n_0)$, set $Y^0 = \frac{1}{1-\lambda} X$, where $\lambda$ is the smallest eigenvalue of $\LA(X)$. This guarantees that $Y^0$ is a boundary point of $\D_A(n_0)$. We then apply \Cref{algo:pure-dilation-algo} with the difference that the algorithm terminates if some $Y^j$ is either a matrix extreme or an Arveson extreme point of $\D_A$.

\end{algo}

\begin{theorem} \label{theorem:algorithm-convergence}
Let $\D_A$ be a bounded free spectrahedron and let $X\in\D_A$. Then, with probability $1$ \Cref{algo:dilation-algorithm} terminates after some finite number $k$ many steps and $Y^k$ is either a matrix extreme or Arveson extreme point of $\D_A$.
\end{theorem}

\begin{proof}
\Cref{algo:pure-dilation-algo} is the subject of \cite[Proposition 2.8]{EFHY21}. In this proposition, the authors show that if $Y^0 \in \D_A(n_0)$ and $A \in \SM{g}{d}$, then with probability one, the algorithm terminates in at most 
$dilDim(Y^0) \leq gn_0$ steps.

\Cref{algo:dilation-algorithm} is almost identical to 
\Cref{algo:pure-dilation-algo} with the only differences between the two algorithms being the construction of $Y^0$ from the initial point and the termination of \Cref{algo:dilation-algorithm} in the case that some $Y^j$ is matrix extreme. Once $Y^0$ is constructed, we apply \Cref{algo:pure-dilation-algo} and the theory described in paragraph one gives us a bound on the number of steps with probability one. Moreover, if \Cref{algo:dilation-algorithm} terminates after $k$ steps, then either $Y^k$ is matrix extreme or $Y^k$ is Arveson extreme.
\end{proof}

 \begin{remark}
 While termination in theory occurs with probability 1,
 in practice 
 this results in a very high but not perfect success rate,
 see \Cref{sec:experiments}, \ref{sec:free-carath}.
 \end{remark}

Given some $Y^j$ as above, $Y^{j+1}$ can be computed by solving two semidefinite programs. The first optimization computes the $c_{j+1}$ and the second computes the $\widehat{\gamma}^{j+1}$. These semidefinite programs can  be solved numerically. The introduction of numerical error into the problem requires some consideration. The first main consideration it that a given dilation step can fail in the sense that the dilation is not a maximal 1-dilation. In this case, we discard this failed dilation step and try again with a different $\beta^{j+1}$. The $\beta^{j+1}$ in question is chosen by first computing a basis for the dilation subspace of $X$ with respect to $\D_A$ and then taking a random convex combination of the basis vectors. Thus unless the dilation subspace dimension of $X$ is one, we expect with probability $1$ that the newly generated $\beta^{j+1}$ will not be a scalar multiple of the original.

In order to prevent an infinite loop, we impose a limit on the number of failed dilation attempts on a single point to some maximum number. For the experiments below, this maximum was taken to be ten, meaning a point could fail to dilate ten times before the process was aborted. This number was chosen to be low as the importance of any particular point out of 10,000 trials is relatively low. In cases where it is important to dilate a particular point of interest to an Arveson or matrix extreme points, it may be appropriate to take a maximum which is much higher.

Another aspect of the algorithm is that there is a $\gamma \in \SM{g}{1}$ that is generated when we compute $c_{j+1}$, but this $\gamma$ is ``thrown out'' and replaced with $\widehat{\gamma}_{j+1}$ in the $Y^{j+1}$. In our experiments described below, we omit this second step and keep the original $\gamma$. Importantly, \Cref{theorem:algorithm-convergence} does not guarantee the termination of this modified algorithm, but in practice it has been shown to be effective.

Given a bounded free spectrahedron $\D_A$, we can use \Cref{algo:dilation-algorithm} to generate many extreme points $X \in \D_A(n)$ such that $k_{A,X}$ is relatively small. By producing a large number of extreme points in such a manner and then counting the number of \mnae  points, we can get a sense for how common such extreme points are.

\sssec{What Do We Call Zero?} \label{section:what-do-we-call-zero}

There is another issue resulting from the introduction of numerical error, namely, what does it mean for a point $X$ with floating point entries to be an extreme point. We can think of such an $X$ as a sum $X = \widehat{X} + X^\delta$ where $\widehat{X}$ is an extreme point of $\D_A$ and $X^\delta$ as some small, nonzero error term. Such an $X$ may not even be in $\D_A$, and is likely not an extreme point. The best result we can hope to achieve is for the entries of $X^\delta$ to be very small. Thus, instead of aiming to show that $X$ is extreme, we aim to show that $X$ is close to being extreme in the sense that $X^\delta$ is small. We will call such points \df{extreme candidates}.

There are many similar situations where we must make 
a decision about what to take to be zero and often
a decision about which space is the nullspace of a given matrix. We formalize such decisions using the following definition.

\bs

\begin{definition} \label{algo:zero-calling}
Let the matrix $M \in M_{n\times m}(\R)$ with singular values $\lambda_1\geq \lambda_2 \geq  \cdots \geq \lambda_\ell$, where $\ell = \min(n,m)$.Let $\epsilon_{mag},\epsilon_{gap} > 0$ be given ($\epsilon_{mag}$ will be referred to as the magnitude tolerance and $\epsilon_{gap}$ will be referred to as the gap tolerance throughout this paper). A singular value $\lambda_i$, for $i = 2,3,\dots,\ell,$ is said to be the \df{first numerical zero} of $M$ if all of the following are true
\begin{enumerate}[\rm(1)]
   \item for all $j < i$, $\lambda_j$ is not the first numerical zero
    \item $\lambda_i < \epsilon_{mag}$
    \item $ \lambda_i / \lambda_{i-1} < \epsilon_{gap}.$
\end{enumerate}
In other words, a singular value $\lambda_i$ is the \df{first numerical zero} of $M$ if it is the first singular value to be both smaller than the magnitude tolerance and have a sufficiently large gap between it and the previous singular value.

A singular value $\lambda_j$ is called a \df{numerical zero} if $\lambda_j \leq \lambda_i$, where $\lambda_i$ is the first numerical zero.
Define a function
\begin{align}
\Delta (M,\epsilon_{mag},\epsilon_{gap}) \ \mbox{to be the index of the first numerical zero of $M$ }
\end{align}
if one exists and False otherwise. (The function $\Delta$ is captured by DetermineNull in NCSE).
\end{definition}

We now illustrate this by determining if a point $X$ is an extreme candidate of $\D_A$ using a modified version of \Cref{theorem:Equations}. As an example, suppose we are attempting to show that $X$ is an Arveson extreme candidate. As a reminder, 
a point $X\in \D_A(n)$ is a Arveson extreme point if and only if the only solution to the linear equation
\[
\lA(\beta^T) K_{A,X} = 0
\]
in the unknown $\beta \in M_{n \times 1}(\R)^g$ is $\beta = 0$. 
Thus, the first step in determining if $X$ is an extreme candidate is to compute the nullspace $K_{A,X}$ using some numerical method, see the upcoming subsections for further discussion and for tolerances used for  this. The linear map $\beta \mapsto \lA(\beta^T) K_{A,X}$ has a matrix representation which we will denote $M_{A,X,Arv} $. 
Thus the equation above has a nontrivial solution if and only if $M_{A,X,Arv}$ has a nontrivial nullspace. In our experiments we say that the nullspace of $M_{A,X,Arv}$ is trivial if and only if $\Delta(M_{A,X,Arv},10^{-15},10^{-15}) = False$. We chose this very tight tolerance after significant trial and error.

An important choice is what size eigenvalues of $L_A(X)$
to declare numerically 0,
that is,
determine the nullspace of $\LA(X) \succeq 0$.
We approach this using the function $\Delta$. 
Specify 
gap and magnitude tolerance,
$\eps_{gap} , \eps_{mag}$,
and select
$$
\dim \cK_{A,X}:= d n - \Delta(L_A(X),  \eps_{mag},\eps_{gap}) + 1
$$
where $X \in \SM{g}{n}$ and $A \in \SM{g}{d}$.
Once the dimension is selected we can take $\cK_{A,X}$ to be
%
  the range of a  matrix with floating point entries whose columns are orthonormal  eigenvectors of $\LA(X)$ corresponding to the 
  $\dim \cK_{A,X}$ smallest
  (i.e., numerically zero) eigenvalues of $\LA(X)$.

\sssec{Nullspace Purification}

Improving the numerical accuracy of computing the nullspace  $K_{A,X}$ of $L_A(X)$ can significantly improve the accuracy of the Arveson, Euclidean, and matrix extreme equations. One potential avenue for this is to slightly perturb the extreme candidate $X$ to a tuple $X^\eps$ so that the first numerical zero of $\LA(X^\eps)$ is small. The question then arises as to how to compute such a perturbation.

The following algorithm, which we call \df{Full Nullspace Purification}, is an algorithm that computes such a perturbation by solving a linear program. Linear programs can be solved quickly and to a high degree of accuracy (compared to SDP) which makes this method particularly effective.

\begin{algo}[Full Nullspace Purification] \label{algo:nullspacepurification}

We are given a bounded free spectrahedron $\D_A$, a (numerical) boundary point $X$ of $\D_A$ and 
tolerances\footnote{In the experiments reported here the tolerances are set to $\eps_{mag}=10^{-7}, 
\eps_{gap}=10^{-2}$
and $\epsilon= 10^{-7}$. 
} $\eps_{mag}, \eps_{gap}$
and $\epsilon > 0$.

Compute
$$
\dim \cK_{A,X}:= d n - \Delta(L_A(X),  \eps_{mag},\eps_{gap}) + 1
$$
and compute $\cK_{A,X}$, the corresponding nullspace of $L_A(X)$.
Recall that the nullspace of $\LA(X)$ is nontrivial, since $X$ is a boundary point. 

Let
 $\eta_\epsilon \in \R$, $Y^\epsilon \in \SM{g}{n}$ be the solution to the linear program
\begin{align} \label{eq:np-program}
\begin{split}
    \eta_\epsilon, Y^\epsilon :=& \quad \underset{\eta\in\R \; Y\in\SM{g}{n}}{\mathrm{Minimizer}} \quad \eta \\
    \mathrm{s.t.}   
    &\phantom \quad \max_{i = 1,2,\dots,dn}\left | \left[K_{A,X}^T \LA(X + Y) K_{A,X}\right]_{ii} \right| \leq \eta \\
    &\phantom \quad \|Y\|_{\max} \leq \epsilon
\end{split}
\end{align}
where $\|\cdot\|_{\max} :\SM{g}{n} \to \R$ is defined as $\|W\|_{\max} = \underset{i,j,k}{\max} |W_{kij}|$ and $W_{kij}$ is the $i,j$ entry of the $k$th matrix in the tuple $W$.

Return $X^\eps= X + Y^\eps$.

\end{algo}

Note you may wish to check that  $L(X^\eps)$ is positive semidefinite,
to wit that all its eigenvalues 
are greater than a given tolerance. 
In algorithms where this is repeated in an inner iteration, checking positivity of the final answer may suffice.

In certain cases, such as when using dilations to compute the free Caratheodory expansion of a point, the extreme point $X$ on which we are applying Nullspace Purification may be of the form
\[
X = \begin{pmatrix}
X^0 & \beta \\
\beta^T & \gamma
\end{pmatrix}
\]
where $X^0 \in \SM{g}{n_0}$ is of particular importance. In such a case, it is often undesirable to perturb $X^0$ even slightly, and thus, we employ the following modified algorithm.

\begin{algo}[Frozen Nullspace Purification] \label{algo:nullspacepurificationfreeze}

We are given $\D_A$ be a bounded free spectrahedron and 
\[
X = \begin{pmatrix}
X^0 & \beta \\
\beta^T & \gamma
\end{pmatrix}
\]
on the (numerical) boundary 
 of $\D_A$, where $X^0 \in \SM{g}{n_0}$, $\beta \in M_{n_0\times s}(\R)^g$, $\gamma \in \SM{g}{s}$, and 
tolerances $\eps_{mag}, \eps_{gap}$ and $\epsilon > 0$.

Compute
$$
\dim \cK_{A,X}:= d n - \Delta(L_A(X),  \eps_{mag},\eps_{gap}) + 1
$$
and compute $\cK_{A,X}$, the corresponding nullspace of $L_A(X)$.
Recall that the nullspace of $\LA(X)$ is nontrivial since $X$ is 
a (numerical) boundary point.

Let $\eta_\epsilon \in \R$, $Y^{\epsilon} \in \SM{g}{n}$ be the solution to the linear program
\begin{align} \label{eq:np-program-freeze}
\begin{split}
    \eta_\epsilon, Y^\epsilon :=& \quad \underset{\eta\in\R \; Y\in\SM{g}{n}}{\mathrm{Minimizer}} \quad \eta \\
    \mathrm{s.t.}
    &\phantom \quad \max_{i = 1,2,\dots,dn}\left|\left[K_{A,X}^T \LA(X + Y) K_{A,X}\right]_{ii}\right| \leq \eta \\
    &\phantom \quad \|Y\|_{\max} \leq \epsilon \\
    &\phantom \quad Y = \begin{pmatrix}
    0_{n_0}^g & \tilde{\beta} \\
    \tilde{\beta}^T & \tilde{\gamma}
    \end{pmatrix}.
\end{split}
\end{align}

Return $X^\eps= X + Y^\eps$.

\end{algo}

\bs

\sssec{Observed Algorithm Properties}
While our implementation of Algorithms \ref{algo:nullspacepurification} and \ref{algo:nullspacepurificationfreeze} does not guarantee that the nullspace quality improves and that positivity is maintained, we ran extensive experiments that illustrate this happens in practice. In fact, when applying Nullspace Purification to an extreme candidate $X \in \cD_A$ we observe that 
\[
\Delta(L_A(X),  10^{-7},10^{-2}) = \Delta(L_A(X^\eps),  10^{-11},10^{-11}),
\]
illustrating that we are able to use much tighter nullspace tolerances after Nullspace Purification. The above equality represents a significant improvement in nullspace quality, as the magnitude tolerance $\epsilon_{mag}$ for extreme candidates generated using semidefinite programming alone can rarely be taken below $10^{-8}$. That is, for an extreme candidate $X$ that is the output of an SDP, one will frequently have
\[
\Delta(L_A(X),10^{-8},10^{-2}) = False.
\]

We note that there are many obvious alternatives to Algorithms \ref{algo:nullspacepurification} and \ref{algo:nullspacepurificationfreeze}. For example, one could minimize all entries of the matrix $K_{A,X}^T \LA(X + Y) K_{A,X}$. It is easier to provide theoretical guarantees for this alternative approach, which is appealing. However, we experimentally observed that our implementation typically notably outperformed the  variations that we tried.

\subsubsection{Using Nullspace Purification to improve \Cref{algo:pure-dilation-algo} and \Cref{algo:dilation-algorithm}.}

In order to increase the numerical accuracy of the extreme point candidates that we generate using \Cref{algo:pure-dilation-algo} and \Cref{algo:dilation-algorithm}, after each dilation step we applied Full Nullspace Purification to the dilated point. We are running two types of experiments. In the first, we use \Cref{algo:dilation-algorithm} to generate \mnae  points without much consideration for the actual points generated. In these experiments, we apply \Cref{algo:nullspacepurification}, i.e., Full Nullspace Purification, after each dilation step.

In the second set of experiments  we perform, the initial point $X^0$ plays an important role, as we search for a free Caratheodory expansion of that particular point using \Cref{algo:pure-dilation-algo}. Here, we instead use Frozen Nullspace Purification (freezing $X^0$) as in \Cref{algo:nullspacepurificationfreeze}.

As previously mentioned, it is in principle possible that a Nullspace Purification step could produce $\LA(X^\epsilon)$ that is not numerically positive semidefinite. In our experiments, we checked and found that every  final dilation point $X$ satisfied $\LA(X) + 10^{-11} I_{dn} \psd 0$.

The use of Nullspace Purification greatly decreases the rate of failure in our experiments. This is especially true for $g = 2$. It is amusing to note that when running all experiments described in \Cref{sec:experiments}  with identically tight tolerances, but without the use of Nullspace Purification, the test took approximately ten times longer to run, and none of the points succeeded in dilating to any form of extreme point. This is not a fair comparison as the accuracy of SDP is often quite low, but the mere fact that it is not a fair comparison shows the power of 
Nullspace Purification.

The next two sections describe experiments done using these algorithms and also give data describing their success rate and performance.

\section{Experimental behavior of  \mnae  points}
\label{sec:experiments}

We randomly  generated 
various free spectrahedra.
The goal was to see if ``randomly 
generated" boundary points $Y^0$ frequently or seldom dilated to \mnae  points, 
where we used \Cref{algo:dilation-algorithm}
as a tool.
The data produced is summarized in 3 tables and high level observations and speculations are in \Cref{section:conclusion}.
One consequence of this numerical study was that it guided discovery of the exact
$g=3$ example,
\Cref{sec:exactg3ex}.
Indeed this suggested using the parameters $d=4, n=3 $ and the actual  spectrahedron 
we used. 

\ssec{Data From Our Experiments: Guide to the Tables}
\label{sec:mnaGuide}

In our experiments, we consider three different parameters $g$, $d$, and $n_0$ where the defining tuple $A \in \SM{g}{d}$ and the initial point $Y^0$, discussed in \Cref{algo:dilation-algorithm}, is in $\SM{g}{n_0}.$ For every pair $g$ and $n_0$, we generated 10,000 extreme point candidates using \Cref{algo:dilation-algorithm} with Full Nullspace Purification as in \Cref{algo:nullspacepurification}. 
The defining tuples used were randomly generated irreducible tuples $A \in \SM{g}{d}$ where $d = g,g+1,g+2,g+3$. For each defining tuple, 25 points
were generated; 100 defining tuples were generated for every value of $d$ totaling 2,500 points for every $g$, $d$, $n_0$. Moreover, the generated spectrahedra were checked to ensure they were bounded.

Throughout \Cref{sec:experiments}, our experimental data will be presented in tables with the same format as in \Cref{table:g=2-table}. Importantly, the tables describe the properties of the final points generated by the algorithm, not any intermediates. Thus, in the language of \Cref{algo:dilation-algorithm}, if $Y^0 \in \D_A(n_0)$ is the initial boundary point and $Y^k \in \D_A(n_0+k)$ is the final point, only $Y^k$ and not $Y^j$ for any $0\leq j< k$ will be represented in the table. For each table, all of the generated points started at the same level $n_0$, 
and the $n$ given in column $3$ is the level that the points ended at after being dilated.

For any given $g$, $d$, and $n$,
\begin{enumerate}
    \item 
    $\#$MnotA,
    $\#$Euc, and $\#$Arv columns give how many of the points where \mnae , Euclidean extreme, or Arveson extreme point candidates respectively;
    \item
    ArvCT, MatCT column gives the \arvcnt and \extcnt in that order;
    \item
    $K_{A,X}$ columns count the number of points with kernel dimension $k_{A,X} = 1,2,3,4,5,$ and $>5$. The second number in this column is the list of all the dilation subspace dimensions of the initial points; 
    \item
    $\#$Fail column counts the number of points which failed to dilate to a Euclidean extreme point.
\end{enumerate}

In one case, a point was determined to be a \mnae  point candidate, but the kernel dimension was uncertain. This case is marked with a $1*.$

A summary of our conclusions from the experiments can be found in \Cref{section:conclusion}.

\subsection{Matrix extreme points of spectrahedra for $g=2$}\label{section:g=2}

We now narrow our focus to the case where $g = 2$. In this case, we conducted the experiment described above, generating 2,500 extreme point candidates for each $g$, $d$, $n_0$ where $d = 2,3,4,5$. Extreme point candidate generation for $g=2$ has a higher failure rate ($\sim0.7\%$ as opposed to the second highest $0.01\%$) than any other value of $g$ that we tested. 
This higher rate of failure is the result of a large number of failed dilation tries which results in  each dilation step
taking significantly longer with some timing out.

A common theme throughout all of our experiments is that we only find examples of \mnae  point candidates when the \extcnt is strictly less than the \arvcnt. That is, we only find \mnae  point candidates when the minimum kernel size of $\LA(X)$ necessary to be matrix extreme is strictly less than the minimum kernel size necessary to be Arveson extreme. We would expect this behavior if the matrix extreme and Arveson extreme equations were randomly generated, but they do have  some structure.\looseness=-1

We did, however, find parameter ranges where the \extcnt is strictly smaller than the \arvcnt, but no \mnae  point candidates were found. Stronger still, we find parameters, such as $g=2$, $d = 3$, $n = 5$, where we find Euclidean extreme points with kernel dimensions sufficiently large to be called matrix extreme, yet none of these points were determined to be matrix extreme candidates.

\begin{table}[h] 
\small
\resizebox{\columnwidth}{!}{%
\begin{tabular}{ |c|c|c|c|c|c|c|c|c|c|c|c|c|c| } 
\hline 
\multicolumn{14}{|c|}{\text{$g = 2$, Starting $n_0$ = 3,4,5}} \\ 
\hline 
\multirow{2}{*}{$g$} & \multirow{2}{*}{$d$} & \multirow{2}{*}{$n$} & \text{$ \#$Mat} & \multirow{2}{*}{\text{$ \#$Euc}} & \multirow{2}{*}{\text{$ \#$Arv}} & \text{ArvCT,} & \multicolumn{6}{c|}{$K_{A,X}$} & \multirow{2}{*}{\text{$ \#$Fail}}\\ 
\cline{8-13}  
 &  &  & \text{not Arv} &  &  & \text{MatCT} & 1 & 2 & 3 & 4 & 5 & \text{$>$5} & \\ 
\hline 
\multirow{28}{*}{2} & \multirow{7}{*}{2} & 3 & 0 & 0 & 0 & \text{3,3} & 0 & 0 & 0 & 0 & 0 & 0 & 0\\ 
\cline{3-14} 
 &  & 4 & 0 & 0 & 0 & \text{4,4} & 0 & 0 & 0 & 0 & 0 & 0 & 0\\ 
\cline{3-14} 
 &  & 5 & 0 & 2,500 & 2,500 & \text{5,5} & 0 & 0 & 0 & 0 & \text{2,500;4} & 0 & 0\\ 
\cline{3-14} 
 &  & 6 & 0 & 0 & 0 & \text{6,6} & 0 & 0 & \text{1;8} & 0 & \text{2;6} & 0 & 3\\ 
\cline{3-14} 
 &  & 7 & 0 & 2498 & 2498 & \text{7,6} & 0 & 0 & 0 & 0 & 0 & 2498 & 0\\ 
\cline{3-14} 
 &  & 8 & 0 & 0 & 0 & \text{8,7} & 0 & 0 & 0 & 0 & 0 & 2 & 2\\ 
\cline{3-14} 
 &  & 9 & 0 & 2497 & 2497 & \text{9,8} & 0 & 0 & 0 & 0 & 0 & 2497 & 0\\ 
\cline{2-14} 
 & \multirow{7}{*}{3} & 3 & 0 & 0 & 0 & \text{2,2} & \text{4;3} & 0 & 0 & 0 & 0 & 0 & 4\\ 
\cline{3-14} 
 &  & 4 & 0 & 1773 & 1746 & \text{3,3} & \text{9;5} & \text{27;3} & \text{1746;3} & 0 & 0 & 0 & 9\\ 
\cline{3-14} 
 &  & 5 & 0 & 680 & 432 & \text{4,3} & \text{18;7} & \text{14;5} & \text{248;3,5} & \text{432;3} & 0 & 0 & 32\\ 
\cline{3-14} 
 &  & 6 & 0 & 2239 & 2194 & \text{4,4} & 0 & \text{26;7} & \text{45;5,7} & \text{796;3,5} & \text{1398;3,5} & 0 & 26\\ 
\cline{3-14} 
 &  & 7 & 0 & 1511 & 1464 & \text{5,4} & 0 & 0 & \text{36;7} & \text{47;5,7} & \text{1464;5,7} & 0 & 36\\ 
\cline{3-14} 
 &  & 8 & 419 & 1061 & 626 & \text{6,5} & 0 & 0 & 0 & \text{16;7} & \text{419;5,7} & 626 & 0\\ 
\cline{3-14} 
 &  & 9 & 0 & 79 & 74 & \text{6,5} & 0 & 0 & 0 & 0 & \text{5;7} & 74 & 0\\ 
\cline{2-14} 
 & \multirow{7}{*}{4} & 3 & 0 & 6 & 0 & \text{2,2} & \text{6;2} & 0 & 0 & 0 & 0 & 0 & 0\\ 
\cline{3-14} 
 &  & 4 & 0 & 2494 & 2494 & \text{2,2} & \text{19;4} & \text{956;2} & \text{1537;2} & 0 & 0 & 0 & 19\\ 
\cline{3-14} 
 &  & 5 & 0 & 1816 & 1786 & \text{3,3} & \text{19;6} & \text{30;4} & \text{1786;4} & 0 & 0 & 0 & 19\\ 
\cline{3-14} 
 &  & 6 & 0 & 2409 & 2372 & \text{3,3} & 0 & \text{37;6} & \text{1998;4,6} & \text{374;4} & 0 & 0 & 0\\ 
\cline{3-14} 
 &  & 7 & 0 & 498 & 447 & \text{4,3} & 0 & 0 & \text{51;6} & \text{447;6} & 0 & 0 & 0\\ 
\cline{3-14} 
 &  & 8 & 0 & 239 & 239 & \text{4,4} & 0 & 0 & 0 & \text{119;6} & \text{120;6} & 0 & 0\\ 
\cline{3-14} 
 &  & 9 & 0 & 0 & 0 & \text{5,4} & 0 & 0 & 0 & 0 & 0 & 0 & 0\\ 
\cline{2-14} 
 & \multirow{7}{*}{5} & 3 & 0 & 519 & 0 & \text{2,2} & \text{519;1} & 0 & 0 & 0 & 0 & 0 & 0\\ 
\cline{3-14} 
 &  & 4 & 0 & 2006 & 1981 & \text{2,2} & \text{25;3} & \text{828;1} & \text{1153;1} & 0 & 0 & 0 & 0\\ 
\cline{3-14} 
 &  & 5 & 0 & 2475 & 2475 & \text{2,2} & \text{59;5} & \text{727;3} & \text{1748;3} & 0 & 0 & 0 & 59\\ 
\cline{3-14} 
 &  & 6 & 0 & 1823 & 1724 & \text{3,3} & 0 & \text{99;5} & \text{1724;5} & 0 & 0 & 0 & 0\\ 
\cline{3-14} 
 &  & 7 & 0 & 618 & 618 & \text{3,3} & 0 & 0 & \text{266;5} & \text{352;5} & 0 & 0 & 0\\ 
\cline{3-14} 
 &  & 8 & 0 & 0 & 0 & \text{4,3} & 0 & 0 & 0 & 0 & 0 & 0 & 0\\ 
\cline{3-14} 
 &  & 9 & 0 & 0 & 0 & \text{4,3} & 0 & 0 & 0 & 0 & 0 & 0 & 0\\ 
\hline 
\end{tabular}
} 
\normalsize
\caption{}
\label{table:g=2-table}
\end{table}

\ssec{Matrix Extreme Points of Spectrahedra for $g=3$}
\label{sec:g-3-dilation-experiment}

We conducted the experiments analogous to those in \Cref{section:g=2} for defining tuples $A\in \SM{3}{d}$ and $d=3,4,5,6.$ As in $g=2$, for parameters where the \extcnt is equal to the \arvcnt, we do not find any \mnae  point candidates. Aside from this, there are some notable distinctions between the $g=2$ and $g=3$ cases. 

Firstly, we can see that only three points failed to dilate to extreme for $g=3$ compared to 209 points for $g=2$. Due to the relative cleanness of the data, we can see that for a given $d$ and $n$, we usually find points with only one of two different kernel dimensions. The failed points in the $g=2$ data made it difficult for this phenomenon to be seen.

Secondly, when $g=3$ 
there is a much larger array of parameters at which we find \mnae  point candidates. With the exception of $g = 3$, $d = 4$, $n = 5$ (and higher values of $n$ which no points dilated to), for every $g$, $d$, and $n$ where the \extcnt is strictly less than the \arvcnt we find \mnae  point candidates. The $d = 4$, $n = 5$ exception may be due to experimental design as all the points that started at $n_0 = 2$ and $n_0 = 3$ only needed one dilation step to reach Arveson or matrix extreme, and all the points that started at $n_0 = 4$ needed at least $3$.

\begin{table}[h] 
\small
\resizebox{\columnwidth}{!}{%
\begin{tabular}{ |c|c|c|c|c|c|c|c|c|c|c|c|c|c| } 
\hline 
\multicolumn{14}{|c|}{\text{$g = 3$, Starting $n_0$ = 2,3,4}} \\ 
\hline 
\multirow{2}{*}{$g$} & \multirow{2}{*}{$d$} & \multirow{2}{*}{$n$} & \text{$ \#$Mat} & \multirow{2}{*}{\text{$ \#$Euc}} & \multirow{2}{*}{\text{$ \#$Arv}} & \text{ArvCT,} & \multicolumn{6}{c|}{$K_{A,X}$} & \multirow{2}{*}{\text{$ \#$Fail}}\\ 
\cline{8-13}  
 &  &  & \text{not Arv} &  &  & \text{MatCT} & 1 & 2 & 3 & 4 & 5 & \text{$>$5} & \\ 
\hline 
\multirow{28}{*}{3} & \multirow{7}{*}{3} & 2 & 0 & 0 & 0 & \text{2,2} & 0 & 0 & 0 & 0 & 0 & 0 & 0\\ 
\cline{3-14} 
 &  & 3 & 0 & 2,500 & 2,500 & \text{3,3} & 0 & 0 & \text{2,500;3} & 0 & 0 & 0 & 0\\ 
\cline{3-14} 
 &  & 4 & 0 & 1 & 0 & \text{4,4} & 0 & 0 & \text{1;6} & 0 & 0 & 0 & 0\\ 
\cline{3-14} 
 &  & 5 & 0 & 2499 & 2499 & \text{5,4} & 0 & 0 & 0 & 0 & \text{2499;6} & 0 & 0\\ 
\cline{3-14} 
 &  & 6 & 0 & 8 & 0 & \text{6,5} & 0 & 0 & 0 & 0 & \text{8;9} & 0 & 0\\ 
\cline{3-14} 
 &  & 7 & 0 & 2492 & 2492 & \text{7,6} & 0 & 0 & 0 & 0 & 0 & 2492 & 0\\ 
\cline{3-14} 
 &  & 8 & 0 & 0 & 0 & \text{8,6} & 0 & 0 & 0 & 0 & 0 & 0 & 0\\ 
\cline{2-14} 
 & \multirow{7}{*}{4} & 2 & 0 & 0 & 0 & \text{2,2} & 0 & 0 & 0 & 0 & 0 & 0 & 0\\ 
\cline{3-14} 
 &  & 3 & 261 & 2,500 & 2239 & \text{3,2} & 0 & \text{261;2} & \text{2239;2} & 0 & 0 & 0 & 0\\ 
\cline{3-14} 
 &  & 4 & 0 & 2256 & 2256 & \text{3,3} & 0 & 0 & \text{2256;5} & 0 & 0 & 0 & 0\\ 
\cline{3-14} 
 &  & 5 & 0 & 187 & 187 & \text{4,3} & 0 & 0 & 0 & \text{187;5} & 0 & 0 & 0\\ 
\cline{3-14} 
 &  & 6 & 299 & 2516 & 2217 & \text{5,4} & 0 & 0 & 0 & \text{299;5,8} & \text{2217;5,8} & 0 & 0\\ 
\cline{3-14} 
 &  & 7 & 34 & 34 & 0 & \text{6,4} & 0 & 0 & 0 & 0 & \text{34;8} & 0 & 0\\ 
\cline{3-14} 
 &  & 8 & 0 & 7 & 7 & \text{6,5} & 0 & 0 & 0 & 0 & 0 & 7 & 0\\ 
\cline{2-14} 
 & \multirow{7}{*}{5} & 2 & 0 & 119 & 0 & \text{2,2} & \text{119;1} & 0 & 0 & 0 & 0 & 0 & 0\\ 
\cline{3-14} 
 &  & 3 & 0 & 2381 & 2381 & \text{2,2} & 0 & \text{428;1} & \text{1953;1} & 0 & 0 & 0 & 0\\ 
\cline{3-14} 
 &  & 4 & 398 & 2,500 & 2102 & \text{3,2} & 0 & \text{397;4} & \text{2102;4} & 0 & 0 & 0 & 0\\ 
\cline{3-14} 
 &  & 5 & 0 & 2165 & 2163 & \text{3,3} & 0 & \text{2;7} & \text{2163;7} & 0 & 0 & 0 & 0\\ 
\cline{3-14} 
 &  & 6 & 64 & 335 & 271 & \text{4,3} & 0 & 0 & \text{64;7} & \text{271;7} & 0 & 0 & 0\\ 
\cline{3-14} 
 &  & 7 & 0 & 0 & 0 & \text{5,4} & 0 & 0 & 0 & 0 & 0 & 0 & 0\\ 
\cline{3-14} 
 &  & 8 & 0 & 0 & 0 & \text{5,4} & 0 & 0 & 0 & 0 & 0 & 0 & 0\\ 
\cline{2-14} 
 & \multirow{7}{*}{6} & 2 & 0 & 2,500 & 2,500 & \text{1,1} & \text{2,500;0} & 0 & 0 & 0 & 0 & 0 & 0\\ 
\cline{3-14} 
 &  & 3 & 0 & 0 & 0 & \text{2,2} & 0 & 0 & 0 & 0 & 0 & 0 & 0\\ 
\cline{3-14} 
 &  & 4 & 0 & 2,500 & 2,500 & \text{2,2} & \text{3;6} & \text{401;3} & \text{2099;3} & 0 & 0 & 0 & 3\\ 
\cline{3-14} 
 &  & 5 & 361 & 2497 & 2136 & \text{3,2} & 0 & \text{361;6} & \text{2136;6} & 0 & 0 & 0 & 0\\ 
\cline{3-14} 
 &  & 6 & 0 & 0 & 0 & \text{3,3} & 0 & 0 & 0 & 0 & 0 & 0 & 0\\ 
\cline{3-14} 
 &  & 7 & 0 & 0 & 0 & \text{4,3} & 0 & 0 & 0 & 0 & 0 & 0 & 0\\ 
\cline{3-14} 
 &  & 8 & 0 & 0 & 0 & \text{4,3} & 0 & 0 & 0 & 0 & 0 & 0 & 0\\ 
\hline 
\end{tabular}
} 
\normalsize
\caption{}
\label{table:g=3-table}
\end{table}

\newpage

\ssec{Matrix Extreme Points of Spectrahedra for $g=4$}
\label{sec:g-4-dilation-experiment}

We conducted the same experiment described in \Cref{section:g=2} for defining tuples $A\in \SM{4}{d}$ and $d=4,5,6,7.$

Similarly to the $g=3$ case, we find that the only \mnae  points that we find occur when the \arvcnt is strictly larger than the \extcnt. Moreover, at a majority of parameters where we do find the \extcnt to be strictly smaller than the \arvcnt, we do find \mnae points.

The \extcnt for $g=4$, $d=7$ and $n=2$ is equal to $1$ which is unusual. There is a theorem of Helton, Klep and Vol\v ci\v c
\cite[Section 7]{HKV18} that says that kernel $1$ boundary points form an open dense subset of the boundary of irreducible free spectrahedra.
The \arvcnt is strictly greater than $1$, so they cannot be Arveson. This is evidence for the existence of \mnae  points for this parameter range, and indeed we find many.

\begin{table}[h] 
\small
\resizebox{\columnwidth}{!}{%
\begin{tabular}{ |c|c|c|c|c|c|c|c|c|c|c|c|c|c| } 
\hline 
\multicolumn{14}{|c|}{\text{$g = 4$, Starting $n_0$ = 2,3}} \\ 
\hline 
\multirow{2}{*}{$g$} & \multirow{2}{*}{$d$} & \multirow{2}{*}{$n$} & \text{$ \#$Mat} & \multirow{2}{*}{\text{$ \#$Euc}} & \multirow{2}{*}{\text{$ \#$Arv}} & \text{ArvCT,} & \multicolumn{6}{c|}{$K_{A,X}$} & \multirow{2}{*}{\text{$ \#$Fail}}\\ 
\cline{8-13}  
 &  &  & \text{not Arv} &  &  & \text{MatCT} & 1 & 2 & 3 & 4 & 5 & \text{$>$5} & \\ 
\hline 
\multirow{24}{*}{4} & \multirow{6}{*}{4} & 2 & 0 & 0 & 0 & \text{2,2} & 0 & 0 & 0 & 0 & 0 & 0 & 0\\ 
\cline{3-14} 
 &  & 3 & 0 & 2,500 & 2,500 & \text{3,3} & 0 & 0 & \text{2,500;4} & 0 & 0 & 0 & 0\\ 
\cline{3-14} 
 &  & 4 & 0 & 0 & 0 & \text{4,4} & 0 & 0 & 0 & 0 & 0 & 0 & 0\\ 
\cline{3-14} 
 &  & 5 & 0 & 2,500 & 2,500 & \text{5,4} & 0 & 0 & 0 & 0 & \text{2,500;8} & 0 & 0\\ 
\cline{3-14} 
 &  & 6 & 0 & 0 & 0 & \text{6,5} & 0 & 0 & 0 & 0 & 0 & 0 & 0\\ 
\cline{3-14} 
 &  & 7 & 0 & 0 & 0 & \text{7,5} & 0 & 0 & 0 & 0 & 0 & 0 & 0\\ 
\cline{2-14} 
 & \multirow{6}{*}{5} & 2 & 0 & 0 & 0 & \text{2,2} & 0 & 0 & 0 & 0 & 0 & 0 & 0\\ 
\cline{3-14} 
 &  & 3 & 104 & 2,500 & 2396 & \text{3,2} & 0 & \text{104;3} & \text{2396;3} & 0 & 0 & 0 & 0\\ 
\cline{3-14} 
 &  & 4 & 2424 & 2424 & 0 & \text{4,3} & 0 & 0 & \text{2424;7} & 0 & 0 & 0 & 0\\ 
\cline{3-14} 
 &  & 5 & 0 & 70 & 70 & \text{4,3} & 0 & 0 & 0 & \text{70;7} & 0 & 0 & 0\\ 
\cline{3-14} 
 &  & 6 & 1 & 6 & 5 & \text{5,4} & 0 & 0 & 0 & \text{1;7} & \text{5;7} & 0 & 0\\ 
\cline{3-14} 
 &  & 7 & 0 & 0 & 0 & \text{6,4} & 0 & 0 & 0 & 0 & 0 & 0 & 0\\ 
\cline{2-14} 
 & \multirow{6}{*}{6} & 2 & 0 & 0 & 0 & \text{2,2} & 0 & 0 & 0 & 0 & 0 & 0 & 0\\ 
\cline{3-14} 
 &  & 3 & 0 & 2,500 & 2,500 & \text{2,2} & 0 & \text{179;2} & \text{2321;2} & 0 & 0 & 0 & 0\\ 
\cline{3-14} 
 &  & 4 & 0 & 2380 & 2380 & \text{3,3} & 0 & 0 & \text{2380;6} & 0 & 0 & 0 & 0\\ 
\cline{3-14} 
 &  & 5 & 16 & 120 & 104 & \text{4,3} & 0 & 0 & \text{16;6} & \text{104;6} & 0 & 0 & 0\\ 
\cline{3-14} 
 &  & 6 & 0 & 0 & 0 & \text{4,3} & 0 & 0 & 0 & 0 & 0 & 0 & 0\\ 
\cline{3-14} 
 &  & 7 & 0 & 0 & 0 & \text{5,4} & 0 & 0 & 0 & 0 & 0 & 0 & 0\\ 
\cline{2-14} 
 & \multirow{6}{*}{7} & 2 & 2,500 & 2,500 & 0 & \text{2,1} & \text{2,500;1} & 0 & 0 & 0 & 0 & 0 & 0\\ 
\cline{3-14} 
 &  & 3 & 0 & 0 & 0 & \text{2,2} & 0 & 0 & 0 & 0 & 0 & 0 & 0\\ 
\cline{3-14} 
 &  & 4 & 188 & 2,500 & 2312 & \text{3,2} & 0 & \text{188;5} & \text{2311;5} & 0 & 0 & 0 & 0\\ 
\cline{3-14} 
 &  & 5 & 0 & 0 & 0 & \text{3,3} & 0 & 0 & 0 & 0 & 0 & 0 & 0\\ 
\cline{3-14} 
 &  & 6 & 0 & 0 & 0 & \text{4,3} & 0 & 0 & 0 & 0 & 0 & 0 & 0\\ 
\cline{3-14} 
 &  & 7 & 0 & 0 & 0 & \text{4,3} & 0 & 0 & 0 & 0 & 0 & 0 & 0\\ 
\hline 
\end{tabular} 
}
\normalsize
\caption{}
\label{table:g=4-table}
\end{table}

\ssec{Dilating matrix extreme points to Arveson extreme points}
\label{sec:uptoArv}

In our experiments, we produced a modest number of \mnae  points. Here we check to see if the small perturbations of these points,
 resulting from Nullspace Purification, 
dilate to Arveson extreme points. By small perturbation, $\widehat{X}$ of $X$ we mean
\[
\|X - \widehat{X}\|_{\max} \leq dilDim \cdot 10^{-7}
\]
where the norm $\|\cdot\|_{\max}$ is the entry wise maximum as in \Cref{algo:nullspacepurification}.

\bs

\sssec{$g=2$} Of the 438 \mnae  points that were generated in our experiments for $g=2$, 81 failed to dilate to Arveson extreme. Of the points that dilated to Arveson, all but eight dilated to Arveson extreme in one step. Those eight points dilated to Arveson in two steps. Out of the original 30,000 points generated in our $g=2$ experiments, a total of 1,445 points failed to dilate to Arveson extreme.

\sssec{$g=3$} Of the 1417 \mnae  points that were generated in our experiments for $g=3$, 26 failed to dilate to Arveson extreme. Of the points that dilated to Arveson, all but 88 dilated to Arveson extreme in one step. The 88 remaining points dilated in two steps.  Out of the original 30,000 points generated in our $g=3$ experiments, a total of 159 points failed to dilate to Arveson extreme.

\sssec{$g=4$} Of the 5,233 \mnae extreme points for $g=4$, all but 332 dilated to Arveson extreme. All but two points that successfully dilated to Arveson extreme dilated in one step. The remaining two dilated to Arveson extreme in two steps. Out of the original 20,000 points generated in our $g=4$ experiments, a total of 332 points failed to dilate to Arveson extreme.

\ssec{Summary of experimental findings}

Our experiments  explored $g=2,3,4$ with $d$ ranging from $d = g$ to $d = g + 3$ with $n$ varying but never exceeding 8.
We primarily focused on these parameter ranges both as a way of supporting the process of constructing exact examples and because for larger values of $g$, the dilation steps take significantly longer to solve.

From the data we can make a number of empirical  observations
which wild optimism extends to the forthcoming speculation.

Begin by defining  the \df{strict count parameters} $g,d,n$ to be such that the \arvcnt is strictly larger than the \extcnt:
$$
\left\lceil \frac{gn}{d} \right\rceil > \left\lceil \frac{(n+1)(g+1)}{2d} - \frac{1}{nd} \right\rceil
$$
        
        \begin{spec} Consider a generic element $A$ of $SM_d(\R)^g$ for which $\cD_A$ is bounded. 
            For   $g>2$ 
                and $d > g$ and at each set of strict count parameters, there exist \mnae  points. 
        In the converse direction, for
        all $g$ the  \mnae  points can only occur at strict count parameters.\footnote{We would expect this behavior if the equations were randomly generated; however, the systems have structure.} \qed 
                       \end{spec}
  
       \noindent \textbf{Evidence for Speculation.}
        For $g=2,3,4$ we only find \mnae  points at strict count parameters.
        
        Conversely, for $g = 3$ in seven of the eight strict count parameters where we found points, we found \mnae  points; six were found in the experiments described in \Cref{sec:g-3-dilation-experiment} and one was found in the experiments described in \Cref{sec:free-carath} where the algorithm used Frozen 
        Nullspace Purification, freezing $Y^0$ as opposed to the unfrozen $Y^0$ used in this section. 
        For $g = 4$ the same was true for six of seven strict count parameters. These points were all found in the experiment described in \Cref{sec:g-4-dilation-experiment}.     
        \qed
        
        {\it When $g$ equals $d$ we do not find any \mnae  points.}
        In this case, we see that for $g=2,3,4$ all of the points generated at $d=g$ that dilated to Arveson extreme ended at the exact same level $n$. We do not see this phenomenon happening in general for $d > g$. One could speculate that this holds whenever $g=d$; however, we are cautious. For example 
        the relationship between $g,d$ and the boundedness of $\cD_A$ changes when $g$ becomes large. In particular, $d$ can be much smaller than $g$ while still resulting in a bounded $\cD_A$ if $g$ is large. In contrast, if $g=2$ or $3$ and $d<g$, then $\cD_A$ is not bounded.

        Now we give more detail on the experiment outcomes.
        \ben
       
        \item For $g = 2,3,4$, we do find values of $d$ and $n$ where there are \mnae  points.

        \item \textbf{Strict count parameters where we do not find \mnae  points}
        \begin{enumerate}
            \item $g = 2$
            \begin{enumerate}
                \item $d = 3$: $n = 5,7,9$
                \item $d = 4$: $n = 7$
            \end{enumerate}
            
            \item $g = 3$
            \begin{enumerate}
                \item $d = 4$: $n = 5$
            \end{enumerate}
            
            \item $g = 4$
            \begin{enumerate}
                \item $d = 5$: $n = 5$ (only 70 points)
            \end{enumerate}
        \end{enumerate}
        
        \item \textbf{\mnae  points} are only observed at
        \begin{enumerate}
            \item $g=2$
            \ben
                \item $d=3$: $n=8$. 
            \een
            \item $g=3$
            \ben
                \item $d=4$: $n=3,6,7$
                
                \item $d = 4$: $n = 8$ (found in \Cref{sec:free-carath} experiments)
                \item $d=5$: $n=4,6$
                \item $d=6$: $n=5$
            \een
            \item $g=4$
            \ben
                \item $d=5$: $n=3,4,6$
                \item $d=6$:  $n=5$  
                \item $d=7$: $n=2,4$
            \een
                
        \end{enumerate}
        
        \item 
        For $g = 2$, individual tries to make a dilation step  had a high rate of failure. This causes our experiments to
        produce fewer matrix extreme and Arveson extreme points per fixed number of start points.
        
    \een

Aside from these main observations, the following are also of interest.

\begin{enumerate}
    \item \textbf{Starting points.}
    In each of our experiments we produce a point at level $n$ (over which we have no control except for the starting level). Bounds on $n$ are:
    \ben 
    \item 
    for $g=2$ the largest $n$ we found was 9. Here, the starting size was 5; 
    \item
    for $g=3$ the largest $n$ we found was 8. Here, the starting size was 4; 
    \item 
    for $g=4$ the largest $n$ we found was 6. Here, the starting size was 3.
    \een
    
    \item
    {\bf Dilation Dimension.}
    Given $g,d,n$.
    For randomly generated free $g,d$-spectrahedron
    and every random  boundary point  of size $n$
    we found that the dilation dimension is $gn -d$.
    This is what we would expect from a randomly generated linear equation of the same size as
    the Arveson \Cref{eq:arv}.
    However, the Arveson equations have structure involving several random 
    parts, so they are not truly random.
    
\end{enumerate}

\section{Behavior of the Free Caratheodory Expansion Algorithm}
\label{sec:free-carath}

As discussed in \Cref{sec:FreeExtSignificance} a classical Caratheodory expansion
of a point in a convex set in terms of its extreme points
for free convex sets is equivalent to a dilation to an Arveson extreme point. An algorithm for such dilation was presented in \Cref{sec:algorithms}.
Here we focus on extremal dilations of an interior point $Y^0$.
We shall see that \Cref{algo:pure-dilation-algo} using Frozen Nullspace Purification, as in \Cref{algo:nullspacepurificationfreeze}, has a failure rate of 
$< 1$\% for $g = 3,4;$
also the number of dilation steps required is on
average small (e.g. 20 \%) relative to the theoretical 
maximum number of steps required.
We make the informal point that 
Nullspace Purification considerably improves  the performance of our Algorithm for creating dilations.

We did a few experiments on dilating boundary points
and found \Cref{algo:pure-dilation-algo} using  Nullspace Purification with and without Freezing 
as in \Cref{algo:nullspacepurificationfreeze} and 
\Cref{algo:nullspacepurification}. As seen in Section \ref{sec:experiments}, dilating boundary points using Full Nullspace purification has a high ($>99\%$) success rate. In this section we will see that Frozen Nullspace Purification, perhaps unsurprisingly, has a notably lower success rate $(>84\%)$ than Full Nullspace Purification when applied to boundary points. 

While a lower success rate for Frozen Nullspace Purification can be problematic where dilating the {\it exact original tuple $X$}, in a large number of applied/numerical settings the tuple $X$ in question is already a numerical approximation to an underlying ``true" tuple of interest. Further perturbing this numerical approximation by a small amount as in Full Nullspace Purification is unlikely to have a significant impact on the reliability of results. In fact, if one believes that the ``true" tuple which underlies the numerical approximation $X$ is a boundary point, then purifying the nullspace of $L_A(X)$ can be viewed as a denoising step which increases the numerical reliability of the original numerical approximation. Thus, the high success rate of Full Nullspace Purification can be enjoyed in many applied/numerical settings.

\ssec{Behavior of the Dilation Algorithm: Guide to the Tables}
\label{sec:dilGuide}

In our experiments, we consider three different parameters $g$, $d$, and $n_0$ where the defining tuple $A \in \SM{g}{d}$ and the initial point $Y^0 \in \D_A(n_0) \subset \SM{g}{n_0}.$
For each $d = g$, $g+1$, $g+2$, $g+3$, one hundred randomly generated irreducible tuples $A \in \SM{g}{d}$ were used after verifying the boundedness of $\cD_A$. For each defining tuple, 25 initial points
were randomly generated, totaling 2,500 points for every triple $g,d,n_0$.

Our data is summarized in 3 tables and corresponding 3D histograms,
one for $g=2$, for $g=3$ and for $g=4$.
To illustrate 3D plots  from a different viewpoint we also show (in \Cref{fig:g2HistGrid}) twelve cross sections of 
the first one ($g=2$).

Our goal is to determine success
rates for the algorithm in dilating to Arveson or to matrix extreme as well as the size $n$ of the dilation.

\begin{enumerate}
    \item 
    Min (resp. Max, Mean) means the minimum (resp. maximum, mean) number of dilation steps, $n-n_0$, observed in our experiment.
    \item
    We focus on $n - n_0$ as a fraction of the dilation subspace dimension, hence define
\[
\mu(n) := \frac{n - n_0}{g n_0} = \frac{n - n_0}{dilDim}.
\]
Mean $\mu$ (resp. Std. Dev. $\mu$) describes the mean (resp. standard deviation) of $\mu$ for each $g,d$. This is our main measure of the size $n$ of the dilation we obtain.
    \item
    Fail means a failure to dilate to Arveson or matrix extreme. Fails are not counted in the statistics (such as the means, standard deviations, etc.).
    \item
    \Mnae  means the point dilated to matrix extreme, but failed to dilate to Arveson; \mnae  points are not counted in the statistics (such as the means, standard deviations, etc.).
\end{enumerate}

\newpage

\begin{table}[h!]
\begin{tabular}{ |c|c|c|c|c|c|c|c|c|c| } 
\hline 
\multicolumn{10}{|c|}{\text{Interior Point Dilation Statistics}} \\ 
\hline 
\multirow{2}{*}{$g$} & \multirow{2}{*}{$d$} & \multirow{2}{*}{$n_ 0$} & \multicolumn{3}{c|}{$n-n_ 0$} & \multirow{2}{*}{\text{Mean $\mu $}} & \text{Std. Dev.} & \multirow{2}{*}{\text{$ \#$Fail}} & \multirow{2}{*}{\text{$ \#$MnotA}}\\ 
\cline{4-6}  
 &  &  & \text{Min} & \text{Max} & \text{Mean} &  & $\mu$ &  & \\ 
\hline 
\multirow{12}{*}{2} & \multirow{3}{*}{2} & 3 & 3 & 3 & 3. & 0.5 & 0. & 50 & 0\\ 
\cline{3-10} 
 &  & 4 & 4 & 4 & 4. & 0.5 & 0. & 4 & 0\\ 
\cline{3-10} 
 &  & 5 & 5 & 5 & 5. & 0.5 & 0. & 1 & 0\\ 
\cline{2-10} 
 & \multirow{3}{*}{3} & 3 & 2 & 6 & 2.652 & 0.442 & 0.131 & 206 & 16\\ 
\cline{3-10} 
 &  & 4 & 2 & 8 & 2.82 & 0.352 & 0.132 & 141 & 68\\ 
\cline{3-10} 
 &  & 5 & 3 & 7 & 3.81 & 0.381 & 0.1 & 161 & 169\\ 
\cline{2-10} 
 & \multirow{3}{*}{4} & 3 & 1 & 3 & 1.658 & 0.276 & 0.133 & 82 & 0\\ 
\cline{3-10} 
 &  & 4 & 2 & 4 & 2.351 & 0.294 & 0.085 & 89 & 0\\ 
\cline{3-10} 
 &  & 5 & 2 & 5 & 2.667 & 0.267 & 0.08 & 241 & 28\\ 
\cline{2-10} 
 & \multirow{3}{*}{5} & 3 & 1 & 2 & 1.457 & 0.243 & 0.083 & 60 & 0\\ 
\cline{3-10} 
 &  & 4 & 1 & 3 & 1.602 & 0.2 & 0.097 & 128 & 0\\ 
\cline{3-10} 
 &  & 5 & 2 & 4 & 2.189 & 0.219 & 0.048 & 287 & 70\\ 
\hline 
\end{tabular} 
\caption{$g=2$. 
Min (resp. Max) means the minimum (resp. maximum) number of dilation steps.
Fail means fail to dilate to Arveson or \mnae . 
MnotA means the point dilated 
to
matrix extreme, but not to Arveson. Recall a matrix extreme point is irreducible, so if it is Arveson, then it is free extreme.
}
\end{table}

\begin{figure}[h]
    \centering
{
    \centering
    \includegraphics[width =
    0.40\textwidth]{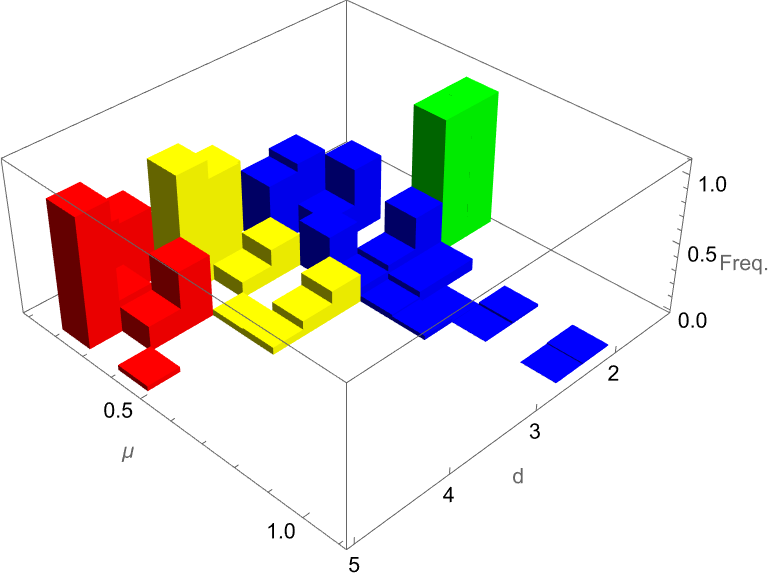}
}
\hspace{0.02\textwidth}
{
    \centering
    \includegraphics[width =
    0.46\textwidth]{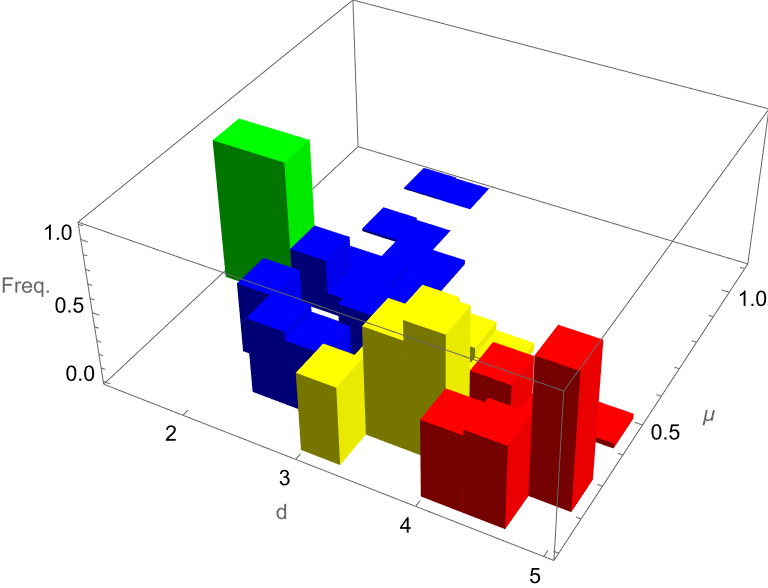}
} \\
\vspace{8pt}
{
    \centering
    \includegraphics[width =
    1\textwidth]{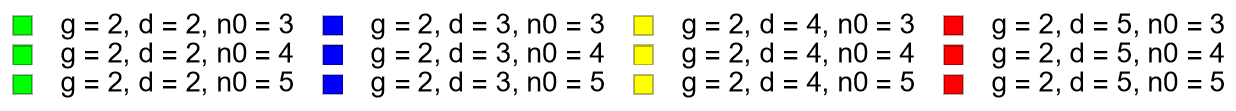}
}
    \caption{$g=2$. Histogram of $\mu$ data, viewed from two different angles.  The plot indicates  frequency of $(\mu,d)$ occurrences.
        }
    \label{fig:g2Hist90Off}
\end{figure}

Perspective on this type of 3D histograms is given by viewing cross sections. We illustrate with one cross section next. Then we show 12 of them.

\begin{figure}[h]
    \centering
    \includegraphics[scale=1.5]{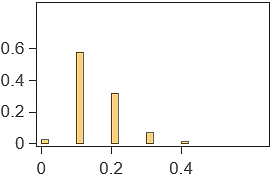}
    \caption{This is a cross section of the 3D histogram for $g=2$ given in \Cref{fig:g2Hist90Off}. In particular, this histogram shows the cross section for $d=2$ and $n_0 = 5.$ The horizontal axis is $0< \mu < 1$ and the vertical axis is the frequency with which a value of $\mu$
    was observed.
    See
    \Cref{fig:g2HistGrid} for the remaining cross sections.
    }
    \label{fig:g2HistGridEntry}
\end{figure}

\begin{figure}[h]
    \centering
    \includegraphics[width=12cm]{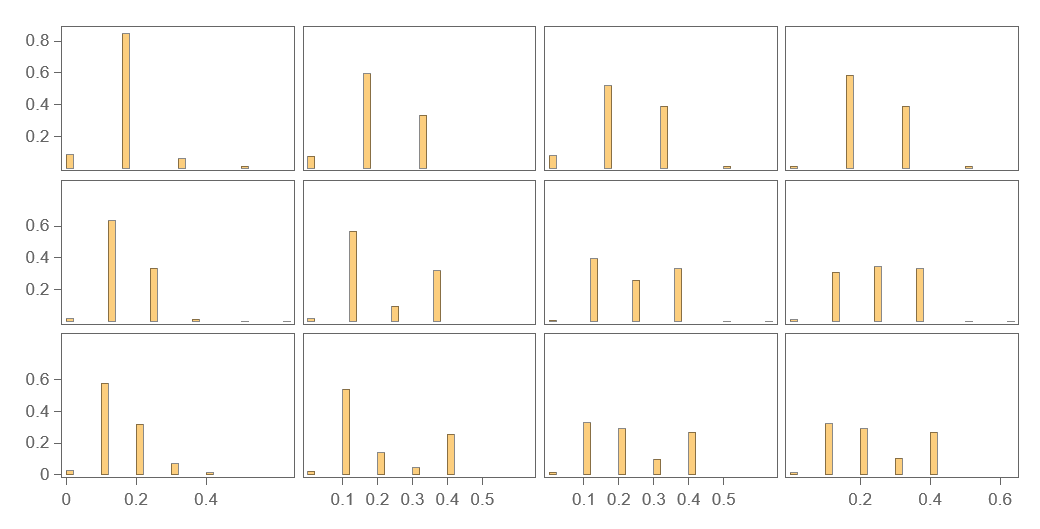}
    \caption{$g=2$. This figure shows a grid of histograms. 
    These are  cross sections of the 3D histogram
    \Cref{fig:g2Hist90Off}.
    The rows correspond to $n_0$, low $n_0=2$ at the
    top thru $n_0=4$ at the bottom, and the columns are $d=2 
    $ on the left
    thru $d=5$ on the right.
    }
    \label{fig:g2HistGrid}
\end{figure}

\bs 
\bs 

$$\mbox{}
$$

\begin{table} [h]
\begin{tabular}{ |c|c|c|c|c|c|c|c|c|c| } 
\hline 
\multicolumn{10}{|c|}{\text{Interior Point Dilation Statistics}} \\ 
\hline 
\multirow{2}{*}{$g$} & \multirow{2}{*}{$d$} & \multirow{2}{*}{$n_ 0$} & \multicolumn{3}{c|}{$n-n_ 0$} & \multirow{2}{*}{\text{Mean $\mu $}} & \text{Std. Dev.} & \multirow{2}{*}{\text{$ \#$Fail}} & \multirow{2}{*}{\text{$ \#$MnotA}}\\ 
\cline{4-6}  
 &  &  & \text{Min} & \text{Max} & \text{Mean} &  & $\mu$ &  & \\ 
\hline 
\multirow{12}{*}{3} & \multirow{3}{*}{3} & 2 & 2 & 2 & 2. & 0.333 & 0. & 6 & 0\\ 
\cline{3-10} 
 &  & 3 & 3 & 3 & 3. & 0.333 & 0. & 5 & 0\\ 
\cline{3-10} 
 &  & 4 & 4 & 4 & 4. & 0.333 & 0. & 9 & 0\\ 
\cline{2-10} 
 & \multirow{3}{*}{4} & 2 & 2 & 5 & 2.047 & 0.341 & 0.041 & 2 & 69\\ 
\cline{3-10} 
 &  & 3 & 2 & 9 & 2.306 & 0.256 & 0.083 & 3 & 15\\ 
\cline{3-10} 
 &  & 4 & 3 & 7 & 3.241 & 0.27 & 0.04 & 1 & 36\\ 
\cline{2-10} 
 & \multirow{3}{*}{5} & 2 & 1 & 3 & 1.328 & 0.221 & 0.101 & 1 & 0\\ 
\cline{3-10} 
 &  & 3 & 2 & 5 & 2.066 & 0.23 & 0.033 & 8 & 0\\ 
\cline{3-10} 
 &  & 4 & 2 & 6 & 2.334 & 0.195 & 0.052 & 7 & 44\\ 
\cline{2-10} 
 & \multirow{3}{*}{6} & 2 & 1 & 2 & 1.26 & 0.21 & 0.073 & 3 & 0\\ 
\cline{3-10} 
 &  & 3 & 1 & 3 & 1.279 & 0.142 & 0.063 & 2 & 1\\ 
\cline{3-10} 
 &  & 4 & 2 & 4 & 2.069 & 0.172 & 0.024 & 10 & 10\\ 
\hline 
\end{tabular} 
\normalsize
\caption{$g=3$. \  
}
\label{table:g3 carath}
\end{table}

\begin{figure*}[h] \label{hist:g3}
\centering
{
    \centering
    \includegraphics[width =
    0.46\textwidth]{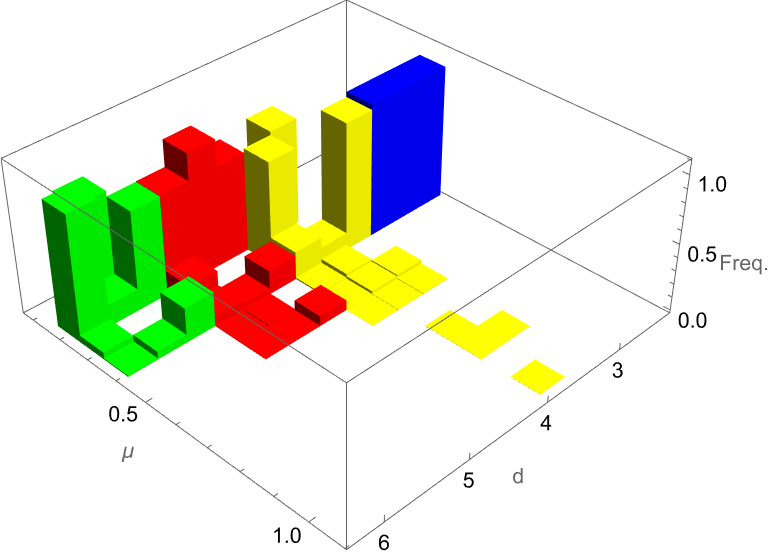}
}
\hspace{0.02\textwidth}
{
    \centering
    \includegraphics[width =
    0.46\textwidth]{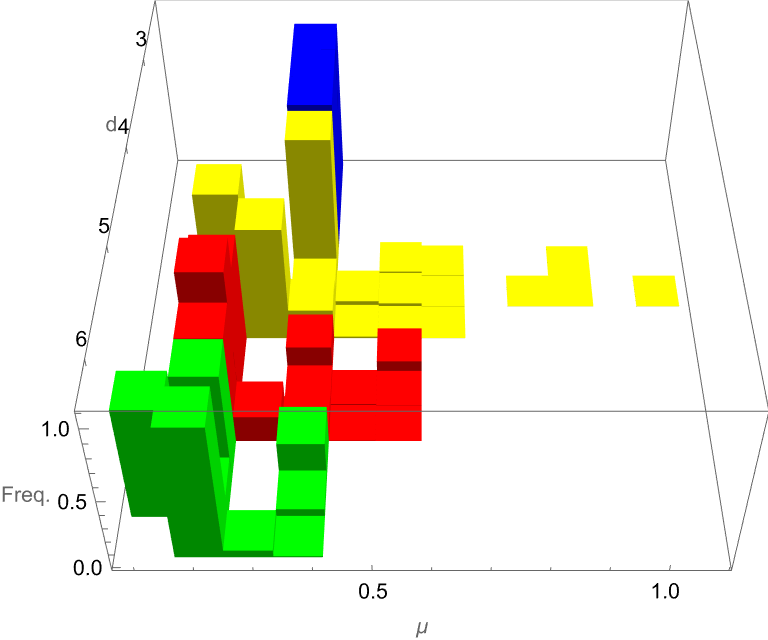}
} \\
\vspace{8pt}
{
    \centering
    \includegraphics[width =
    1\textwidth]{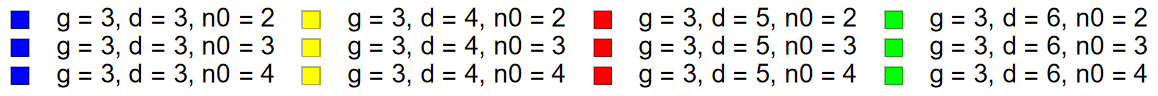}
}
\captionsetup{width=0.98\textwidth}
\caption{
$g=3$. Histogram of $\mu$ data, viewed from two angles.
}
\end{figure*}

\bs
\bs 


\begin{table} [h]
\begin{tabular}{ |c|c|c|c|c|c|c|c|c|c| } 
\hline 
\multicolumn{10}{|c|}{\text{Interior Point Dilation Statistics}} \\ 
\hline 
\multirow{2}{*}{$g$} & \multirow{2}{*}{${d}$} & \multirow{2}{*}{$n_ 0$} & \multicolumn{3}{c|}{$n-n_ 0$} & \multirow{2}{*}{\text{Mean $\mu $}} & \text{Std. Dev.} & \multirow{2}{*}{\text{$ \#$Fail}} & \multirow{2}{*}{\text{$ \#$MnotA}}\\ 
\cline{4-6}  
 &  &  & \text{Min} & \text{Max} & \text{Mean} &  & $\mu$ &  & \\ 
\hline 
\multirow{8}{*}{4} & \multirow{2}{*}{4} & 2 & 2 & 2 & 2. & 0.25 & 0. & 8 & 0\\ 
\cline{3-10} 
 &  & 3 & 3 & 3 & 3. & 0.25 & 0. & 1 & 0\\ 
\cline{2-10} 
 & \multirow{2}{*}{5} & 2 & 2 & 4 & 2.088 & 0.261 & 0.036 & 1 & 17\\ 
\cline{3-10} 
 &  & 3 & 2 & 5 & 2.061 & 0.172 & 0.022 & 0 & 0\\ 
\cline{2-10} 
 & \multirow{2}{*}{6} & 2 & 1 & 4 & 1.11 & 0.139 & 0.045 & 0 & 0\\ 
\cline{3-10} 
 &  & 3 & 2 & 5 & 2.127 & 0.177 & 0.028 & 2 & 3\\ 
\cline{2-10} 
 & \multirow{2}{*}{7} & 2 & 1 & 3 & 1.137 & 0.142 & 0.05 & 1 & 0\\ 
\cline{3-10} 
 &  & 3 & 2 & 4 & 2.014 & 0.168 & 0.01 & 2 & 5\\ 
\hline 
\end{tabular} 
\caption{$g=4$.
}
\label{ta:g4carath}
\end{table}


\begin{figure}[h]
    \centering
{
    \centering
    \includegraphics[width =
    0.46\textwidth]{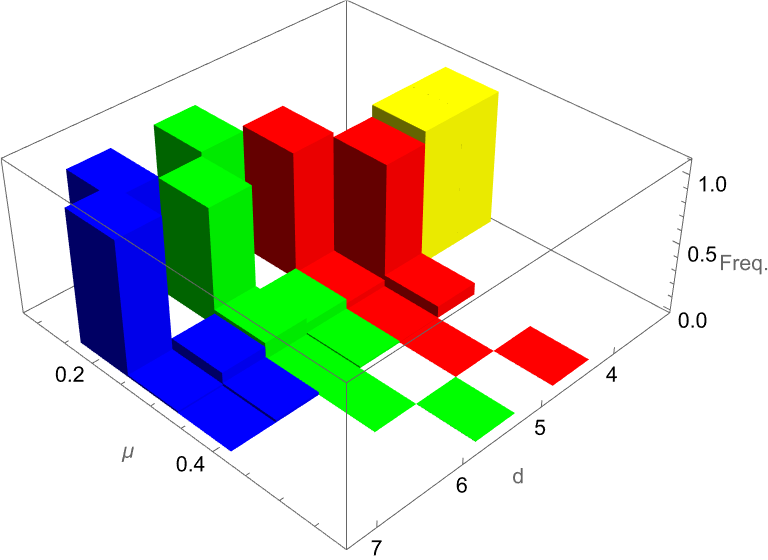}
}
\hspace{0.02\textwidth}
{
    \centering
    \includegraphics[width =
    0.46\textwidth]{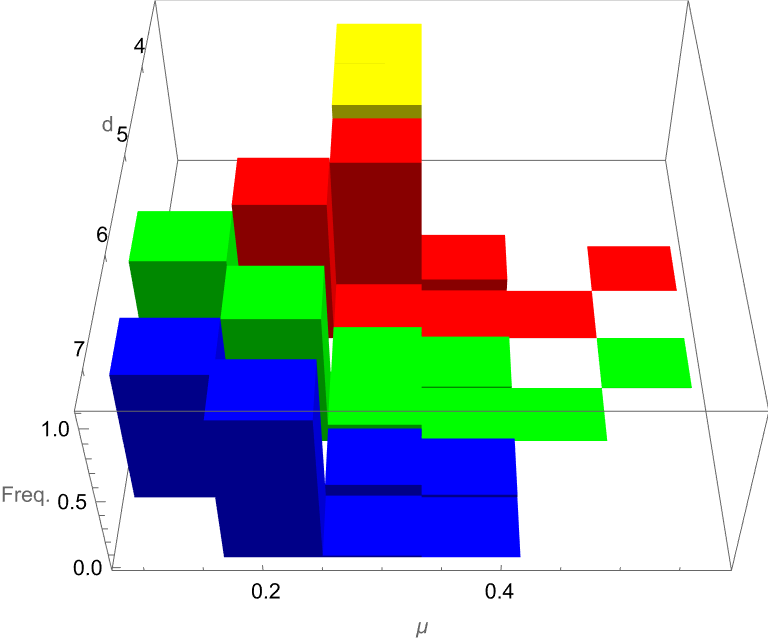}
} \\
\vspace{8pt}
{
    \centering
    \includegraphics[width =
    1\textwidth]{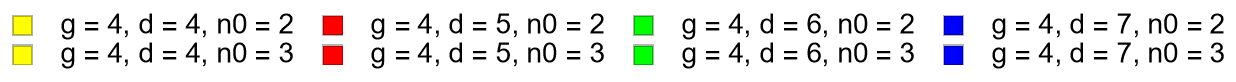}
}
\captionsetup{width=0.98\textwidth}
    \caption{$g=4$. Histogram of $\mu$ data. }
    \label{fig:g4Hist90}
\end{figure}


 \newpage
 \mbox{}
 \newpage

\eject

\ssec{Summary of Findings with Freezing}

\sssec{Success Rate of Our Algorithm}

Our success rate in dilating to Arveson or matrix extreme is summarized below.

\begin{enumerate}
    \item $g = 2$
    \begin{enumerate}
        \item Interior Point Failure Rate: $\sim 4.8\%$.
        
        \item Boundary Point Failure Rate: $\sim 15.4\%$.\\
              Here  $\sim 8.3\%$ of all initial $Y^0$ have dilation dimension $=1$ and \\
              \mbox{} \quad \ \ \  $\sim 15.7\%$ of those  initializations  $Y^0$ which fail. 
    \end{enumerate}
    \item $g = 3$
    \begin{enumerate}
        \item Interior Point Failure Rate: $\sim 0.2\%$
        \item Boundary Point Failure Rate: $\sim 6.4\%$.\\
              Here  $\sim 8.3\%$ of all initial $Y^0$ have dilation dimension $=1$ and \\
              \mbox \quad \ \ \ \
              $\sim 6.8\%$ 
of those  initializations  $Y^0$ which fail. 
    \end{enumerate}
    \item $g = 4$
    \begin{enumerate}
        \item Interior Point Failure Rate: $\sim 0.08\%$
        \item Boundary Point Failure Rate: $\sim 3.9\%$.\\\\
              Here $\sim 8.3\%$ of all initial $Y^0$ have dilation dimension $=1$ and \\
              \mbox{} \quad \ \ \
              $\sim 0\%$ 
              of those  initializations  $Y^0$ which fail. 
    \end{enumerate}
\end{enumerate}

\sssec{Size of dilations are usually small}
\mbox{} 

\ben
\item
 $Min(n-n_0)
 \leq  Mean(n-n_0) < Min(n-n_0)+ \eta_g
$\\
where $\eta_g<1$ for all $g,d,n_0$
and $\eta_g$ gets smaller as $g$ increases.

\item

As $g$ increases the gap between Max and Min decreases.

\item
The worst case that is theoretically possible $\mu$ is $1$, but we uniformly find better behavior.
For example, when $d - g > 1$ we see the following trend.
\begin{enumerate}
    \item
    For $g = 2$, the mean of $\mu$ is less than $0.3$;
    \item
    For $g = 3$, the mean of $\mu$ is less than $0.23$;
    \item
    For $g = 4$, the mean of $\mu$ is less than $0.18$.
\end{enumerate}
\een

\ssec{Size of free extreme dilations and crude 
insight into to $\mu$}

In a more quantitative vein,
a crude estimate for the mean of $\mu$
is
\[ \mu \sim \mu_{est}:= \frac{
\lceil (g n_0) / d \rceil}{g n_0}.
\]
Note that $\mu = \mu_{est}$ is equivalent to  a single dilation step  decreasing  the dilation subspace dimension by $d$, so a main use of this formula is as a test for measuring how often this occurs.
We compared this (see \Cref{table:dilDim})
as well as  the size of   free extreme dilations to the data produced 
in our experiments.
The correspondence of $\mu_{est}$ to our experimental data is very rough and typically overestimates the size dilation we get. However,
the behavior  $\mu = \mu_{est}$
is  observed to always hold  for $g = d$.

If the dilation subspace equations, \Cref{eq:arv}, were generated at random (due to optimizing they are not), one would expect the dilation subspace dimension to be $gn - dk$. 
Hence, if $s$ is the number of steps taken so far,  and we assume that $k = cs$ for some fixed $c$, we obtain the following equation for the dilation subspace dimension
\[
EstDilDim = g(n_0 + s) - d(c s).
\]
At the final iteration we have
$EstDilDim \leq  0$ which yields $g(n_0 + s) - d(c s) \leq 0$. 
We can then solve for the smallest integer $\widehat{s}(c)$ such that the inequality holds, giving us
\[
\widehat{s}(c) = \left\lceil \frac{gn_0}{cd - g}\right\rceil.
\]
We would expect $\mu(n) \leq  \widehat{s}(1) / (gn_0)$ as the kernel dimension must increase by at least one after every step. This turns out to be true in all of our experiments.

\begin{table} 
\begin{tabular}{ |c|c|c|c|c|c|c|c| } 
\hline 
\multicolumn{8}{|c|}{\text{Interior Point Dilation Estimate Tests}} \\ 
\hline 
\multirow{2}{*}{$g$} & \multirow{2}{*}{$d$} & \multirow{2}{*}{$n_0$} & \multicolumn{2}{c|}{\text{Error}} & \multicolumn{3}{c|}{\text{Mean}}\\ 
\cline{4-5} \cline{6-8}  
 &  &  & $\mu _{\text{est}}$ & \% & \text{$\Delta $dilDim/step} & \text{$\Delta $dilDim/step past 1} & \text{first $\Delta $dilDim}\\ 
\hline 
\multirow{12}{*}{2} & \multirow{3}{*}{2} & 3 & 0. & \text{0$\%$} & 2. & 2. & 2.\\ 
\cline{3-8} 
 &  & 4 & 0. & \text{0$\%$} & 2. & 2. & 2.\\ 
\cline{3-8} 
 &  & 5 & 0. & \text{0$\%$} & 2. & 2. & 2.\\ 
\cline{2-8} 
 & \multirow{3}{*}{3} & 3 & 0.109 & \text{24.7$\%$} & 2.432 & 2.015 & 2.75\\ 
\cline{3-8} 
 &  & 4 & -0.023 & \text{-6.53$\%$} & 3.174 & 3.283 & 2.917\\ 
\cline{3-8} 
 &  & 5 & -0.019 & \text{-4.99$\%$} & 2.773 & 2.693 & 2.828\\ 
\cline{2-8} 
 & \multirow{3}{*}{4} & 3 & -0.057 & \text{-20.7$\%$} & 4.442 & 3.075 & 4.2\\ 
\cline{3-8} 
 &  & 4 & 0.044 & \text{15$\%$} & 3.609 & 3.153 & 4.026\\ 
\cline{3-8} 
 &  & 5 & -0.033 & \text{-12.4$\%$} & 4.032 & 3.65 & 4.304\\ 
\cline{2-8} 
 & \multirow{3}{*}{5} & 3 & -0.09 & \text{-37$\%$} & 4.628 & 3. & 4.628\\ 
\cline{3-8} 
 &  & 4 & -0.05 & \text{-25$\%$} & 6.084 & 3.9 & 5.911\\ 
\cline{3-8} 
 &  & 5 & 0.019 & \text{8.68$\%$} & 4.715 & 3.59 & 5.832\\ 
\hline 
\multirow{12}{*}{3} & \multirow{3}{*}{3} & 2 & 0. & \text{0$\%$} & 3. & 3. & 3.\\ 
\cline{3-8} 
 &  & 3 & 0. & \text{0$\%$} & 3. & 3. & 3.\\ 
\cline{3-8} 
 &  & 4 & 0. & \text{0$\%$} & 3. & 3. & 3.\\ 
\cline{2-8} 
 & \multirow{3}{*}{4} & 2 & 0.008 & \text{2.35$\%$} & 2.958 & 1.552 & 4.345\\ 
\cline{3-8} 
 &  & 3 & -0.077 & \text{-30.1$\%$} & 4.137 & 3.782 & 4.434\\ 
\cline{3-8} 
 &  & 4 & 0.02 & \text{7.41$\%$} & 3.768 & 3.41 & 4.446\\ 
\cline{2-8} 
 & \multirow{3}{*}{5} & 2 & -0.112 & \text{-50.7$\%$} & 5.164 & 3.413 & 4.985\\ 
\cline{3-8} 
 &  & 3 & 0.008 & \text{3.48$\%$} & 4.41 & 2.801 & 5.997\\ 
\cline{3-8} 
 &  & 4 & -0.055 & \text{-28.2$\%$} & 5.411 & 4.714 & 6.051\\ 
\cline{2-8} 
 & \multirow{3}{*}{6} & 2 & 0.043 & \text{20.5$\%$} & 5.22 & 3. & 5.22\\ 
\cline{3-8} 
 &  & 3 & -0.08 & \text{-56.3$\%$} & 7.923 & 5.184 & 7.683\\ 
\cline{3-8} 
 &  & 4 & 0.005 & \text{2.91$\%$} & 5.869 & 4.14 & 7.573\\ 
\hline 
\multirow{8}{*}{4} & \multirow{2}{*}{4} & 2 & 0. & \text{0$\%$} & 4. & 4. & 4.\\ 
\cline{3-8} 
 &  & 3 & 0. & \text{0$\%$} & 4. & 4. & 4.\\ 
\cline{2-8} 
 & \multirow{2}{*}{5} & 2 & 0.011 & \text{4.21$\%$} & 3.883 & 2.037 & 5.75\\ 
\cline{3-8} 
 &  & 3 & -0.078 & \text{-45.3$\%$} & 5.884 & 5.919 & 5.812\\ 
\cline{2-8} 
 & \multirow{2}{*}{6} & 2 & -0.111 & \text{-79.9$\%$} & 7.595 & 5.603 & 7.419\\ 
\cline{3-8} 
 &  & 3 & 0.01 & \text{5.65$\%$} & 5.747 & 3.987 & 7.519\\ 
\cline{2-8} 
 & \multirow{2}{*}{7} & 2 & -0.108 & \text{-76.1$\%$} & 7.506 & 4.572 & 7.416\\ 
\cline{3-8} 
 &  & 3 & 0.001 & \text{0.595$\%$} & 5.972 & 2.686 & 9.25\\ 
\hline 
\end{tabular} 
\caption{
    Table describes how accurately our rough ``estimates'' are for $\mu$. Error $(\mu_{est})$ is $\mu - \mu_{est}$, so this being negative means that the ultimate dilation is smaller than we would expect based on $\mu_{est}$. \\
    The signed percent error is calculated using $(\mu - \mu_{est}) / \mu$. \\
    $\Delta$dilDim/step is the average of how much the dilation dimension decreases per step.
}
\label{table:dilDim}
\end{table}

\newpage

\newpage

\subsubsection{
For each $g = d$ the function $\mu(n)$ is observed to be independent of $n$}

In our data, we observe that when $g=d$, the dilation dimension decreases by exactly $g$ with each maximal $1$-dilation; that is, $\mu(n) = 1/g$. One could reasonably speculate that this trend holds for all spectrahedra for which $g=d$. However, we are cautious in making such a conjecture. 

Our caution stems from the fact that, in a sense, there are relatively few different spectrahedra which satisfy $g=d$ when $g$ is ``small". For example, we saw in Section \ref{ssec:d=2proof} that, up to invertible projective maps, there is only one bounded free spectrahedron with $g=d=2$. It is therefore difficult to say if this behavior is due to the fact that $g=d$ or if it is an artifact which arises when $g=d$ and $g$ is small.

\section{Conclusions} \label{section:conclusion}

In this article we proved the existence of matrix extreme points which are not free extreme for real free spectrahedra. To accomplish this, we produced algorithms that reliably compute, using algebraic arithmetic,
\mnae  points when $g = 3$.
In addition when $g=4$ we have produced ``by hand" a \mnae  point. Furthermore, using theory of projective maps of free spectrahedra, we showed that if $d=2$, then matrix extreme points and free extreme points are the same.

In addition to the above theoretical results, we have ``perfected" a numerical algorithm for producing
\mnae  candidates as described in \Cref{subsection:extreme-point-generation}. The reliability of this algorithm relies on a new technique we call 
Nullspace Purification. This modified algorithm has a much lower rate of failure to the unmodified algorithm used previously \cite{EH19}, especially in the $g = 2$ case. The Nullspace Purification algorithm yields extreme point candidates with numerical accuracy on the order of $10^{-13}$ as opposed to the $10^{-7}$ accuracy that we find using semidefinite programming alone which allows us to use tighter tolerances when determining if a point is an extreme candidate.

The experiments described in \Cref{sec:experiments} 
yield the following interesting results. Firstly and most importantly, for $g = 2,3,4$ we find \mnae  candidates. Secondly, we never find any \mnae  candidates when $g = d$. Thirdly, we only find \mnae  candidates when the \arvcnt is strictly greater than the \extcnt. Conversely, for $g = 3,4$ and $d > g$, in thirteen of sixteen cases, when the \arvcnt $\left \lceil \frac{gn}{d} \right \rceil$ is larger than the \extcnt $\left \lceil \frac{g n(n+1) - 2}{2dn} \right \rceil$, we find \mnae  candidates. The three exceptions are potentially explainable through experimental design. We speculate that for $g = 3,4$ and $d > 4$, if the \arvcnt is strictly greater than the \extcnt, then there is a \mnae  for that value of $g, d, n$.

Dilating an interior point in a free spectrahedron to an Arveson extreme point
numerically was done using
\Cref{algo:pure-dilation-algo} with Frozen Nullspace Purification 
and has 
success rates of
$$
\mbox{
 $95.2\%$ for $g=2$, \qquad $99.8\%$ for $g=3$, \qquad $99.92\%$ for $g=4$.
 }
 $$
 Also, the dilation size was
investigated in \Cref{sec:free-carath}.
The ratio $\mu(n)$ of the observed dilation step size required to produce an Arveson dilation divided by the theoretical maximum required dilation step
satisfies
$$ 
\mbox{
mean$[\mu] <0.3$ \  \ for $g=2$, \qquad \ mean$[\mu] <0.23$ \ \ for $g=3$,
}
$$
$$
\mbox{
\ mean$[\mu] <0.18$ \ \ for $g=4$.
}
$$
These
bounds on the mean of $\mu$ show that an Arveson dilation can typically be produced 
more efficiently than 
the theoretical maximum required dilation step (which is not itself big)
and suggests this effect
improves 
 with increasing $n$.

\newpage

\bibliographystyle{alpha}
\bibliography{MatConvexRefs}

\newpage

\newpage

\section{Online Appendix:
Tolerances and the definition of 0} \label{section:appendix}

In our floating point experiments, it was often necessary to make a call about whether a small floating point number is ``zero''. This problem primarily arises in the context of singular values and eigenvalues when computing the nullspace of a matrix. 
This was mentioned  earlier in \Cref{section:what-do-we-call-zero} as well as the 
importance of  
the magnitude of  gap tolerances;
as was said there they were chosen
to be  $\eps_{gap} = \eps_{mag} = 10^{-11}$ for determining the nullspace of $\LA(X)$ and
$\eps_{gap} = \eps_{mag} = 10^{-15}$ for the extreme equations in our experiments.

This section gives more detail and describes experiments which motivated this choice.

\subsection{The Effects of Different Tolerances}

Crucial to all the tests conducted in the paper is the choice of the magnitude and gap tolerances. For instance, if one were to choose a magnitude tolerance of $10^{-100}$ for finding the nullspace of $L_A(X)$, then we will conclude $L_A (X)$ does not have a nullspace for any of the $X$ we generate. While $10^{-100}$ is clearly too tight of a tolerance for use with floating point numbers, around $10^{-12}$ we run into some grey areas.

When testing a point to see if it is extreme, there are two sets of tolerances that we consider: the LMI Tolerances and the Extreme Equation Tolerances (EE Tolerances). The LMI Tolerances are used to determine whether the nullspace of $\LA(X)$ is sufficiently large for $X$ to be considered extreme and the EE Tolerances are used to determine whether the extreme equation has a solution as discussed in \Cref{sec:algorithms}. It is important note to that both of these tolerances are used to determine the size of a certain nullspace, but in order to be called extreme the point must have a large LMI nullspace and no extreme equation solution. Thus, tightening the LMI Tolerances will result in fewer points being called extreme as the apparent LMI nullspace will be smaller and tightening the System Tolerances will result in more points being called extreme as the extreme equation is less likely to be said to have a solution. As a result of this, these two tolerances must be considered independently. Below are some tables showing the results of running the same $g=3$ test as in Section 3.3 but with different LMI and EE Tolerances.

\subsubsection{EE Tolerance $10^{-10}$}

For the following experiments, we fixed the EE Tolerances to $10^{-10}$ and varied the LMI Tolerance from $10^{-10}$ to $10^{-13}$.

\begin{figure}[H]
    \centering
    \includegraphics[width=12cm]{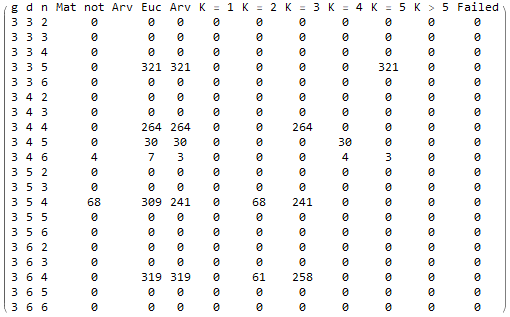}
    \caption{$g = 3$. Starting $n_0 = 3$, $10^{-10}$ LMI Tolerance, $10^{-10}$ EE Tolerance, 1250 Test Points}
    \label{fig:log_singular_values}
\end{figure}

\begin{figure}[H]
    \centering
    \includegraphics[width=12cm]{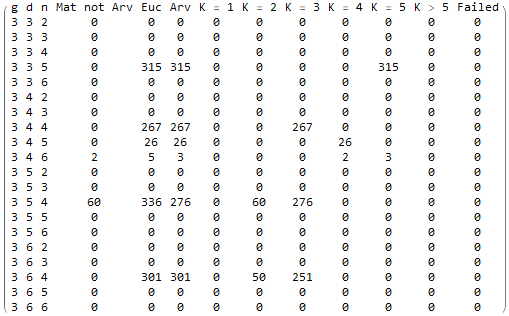}
    \caption{$g = 3$, Starting $n_0 = 3$, $10^{-11}$ LMI Tolerance, $10^{-10}$ EE Tolerance, 1250 Test Points}
    \label{fig:log_singular_values}
\end{figure}

\begin{figure}[H]
    \centering
    \includegraphics[width=12cm]{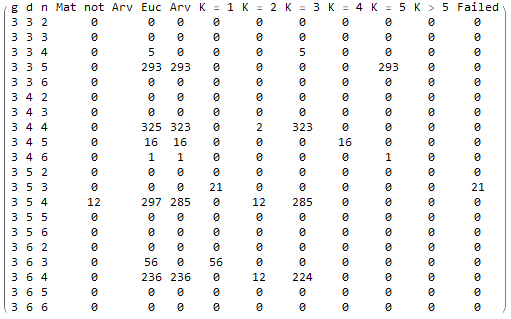}
    \caption{$g = 3$, Starting $n_0 = 3$, $10^{-12}$ LMI Tolerance, $10^{-10}$ EE Tolerance, 1250 Test Points}
    \label{fig:log_singular_values}
\end{figure}

\begin{figure}[H]
    \centering
    \includegraphics[width=12cm]{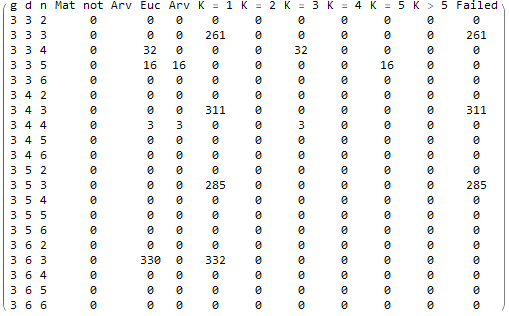}
    \caption{$g = 3$, Starting $n_0 = 3$, $10^{-13}$ LMI Tolerance, $10^{-10}$ EE Tolerance, 1250 Test Points}
    \label{fig:log_singular_values}
\end{figure}

It is important to note that there is minimal change when tightening the tolerance from $10^{-10}$ to $10^{-11}$ in the number of \mnae{} points. However, tightening further has a dramatic effect, causing the number \mnae{} points to drop precipitously. This can be explained by the sharp rise in ``Failed" points, namely points that never succeeded in dilating.

\subsubsection{EE Tolerance $10^{-15}$}

For the following experiments, we fixed the EE Tolerances to $10^{-15}$ and varied the LMI Tolerance from $10^{-10}$ to $10^{-13}$.

\begin{figure}[H]
    \centering
    \includegraphics[width=12cm]{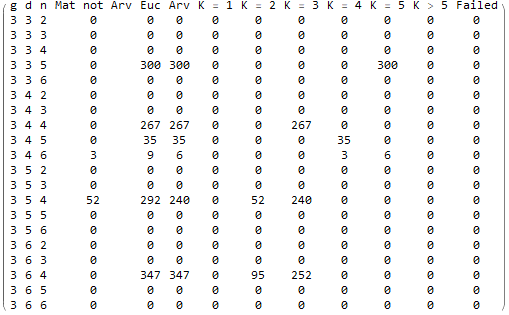}
    \caption{$g = 3$, Starting $n_0 = 3$, $10^{-10}$ LMI Tolerance, $10^{-15}$ EE Tolerance, 1250 Test Points}
    \label{fig:log_singular_values}
\end{figure}

\begin{figure}[H]
    \centering
    \includegraphics[width=12cm]{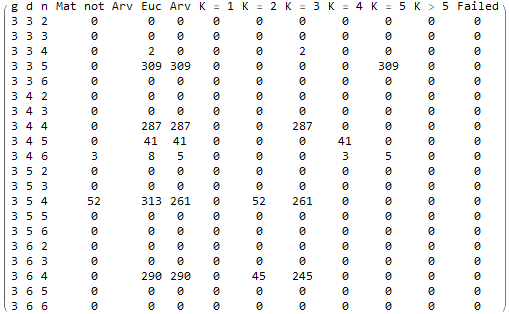}
    \caption{$g = 3$, Starting $n_0 = 3$, $10^{-11}$ LMI Tolerance, $10^{-15}$ EE Tolerance, 1250 Test Points}
    \label{fig:log_singular_values}
\end{figure}

\begin{figure}[H]
    \centering
    \includegraphics[width=12cm]{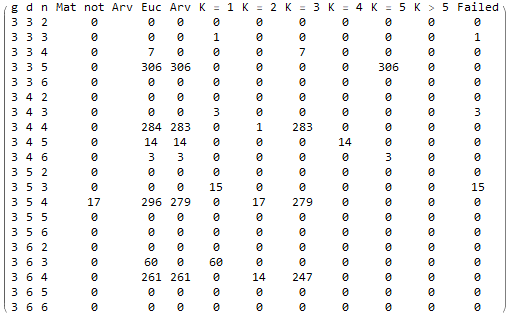}
    \caption{$g = 3$, Starting $n_0 = 3$, $10^{-12}$ LMI Tolerance, $10^{-15}$ EE Tolerance, 1250 Test Points}
    \label{fig:log_singular_values}
\end{figure}

\begin{figure}[H]
    \centering
    \includegraphics[width=12cm]{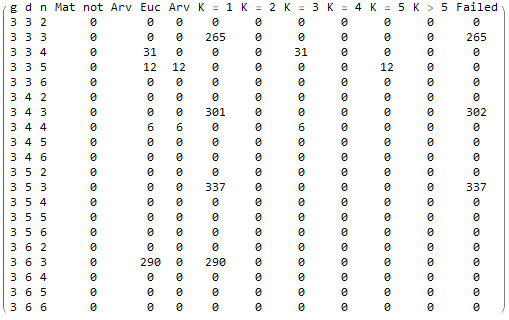}
    \caption{$g = 3$, Starting $n_0 = 3$, $10^{-13}$ LMI Tolerance, $10^{-15}$ EE Tolerance, 1250 Test Points}
    \label{fig:log_singular_values}
\end{figure}

Comparing the $10^{-10}$ and $10^{-15}$ EE Tolerance runs, we see that changing the EE Tolerance, even very dramatically, does not have a significant effect on the number of \mnae{} points. This is perhaps unsurprising as typically when a point is found to not be extreme, it is because the dimension of the nullspace of $L_A(X)$ is not sufficiently large. This causes the extreme equation to have more unknowns than equations, hence the extreme equation necessarily has an exact zero. Making the EE tolerances smaller has no effect on the classification of such points.

In light of the experimental results in this section, we set the EE Tolerance to $10^{-15}$ and the LMI Tolerance to $10^{-11}$. These are essentially the tightest tolerances that we can realistically set as a $10^{-16}$ tolerance for the EE Tolerance would run up against machine precision and $10^{-11}$ is the tightest LMI Tolerance we can set given the empirical testing presented above which shows the limited numerical accuracy of points coming from an SDP and then Full Nullspace Purification.

\end{document}